\documentclass[11pt]{article}

\usepackage[english]{babel}
\usepackage{amsmath,amssymb,amsthm}
\usepackage{enumerate}
\usepackage{verbatim}
\usepackage{bm}
\usepackage{dsfont}
\usepackage{xcolor}
\usepackage{mathtools}

\newtheorem{theorem}{Theorem}[section]
\newtheorem{proposition}[theorem]{Proposition}
\newtheorem{lemma}[theorem]{Lemma}
\newtheorem{corollary}[theorem]{Corollary}

\theoremstyle{definition}
\newtheorem{definition}[theorem]{Definition}

\newtheorem{assumption}[theorem]{Assumption}
\newtheorem{remark}[theorem]{Remark}

\newcommand{\Z}{\mathbb{Z}}
\newcommand{\N}{\mathbb{N}}
\newcommand{\R}{\mathbb{R}}
\newcommand{\C}{\mathbb{C}}
\newcommand{\D}{\mathbb{D}}

\newcommand{\bd}{\partial \D}
\newcommand{\bD}{\partial \D}

\newcommand{\Hil}{\mathcal{H}}
\newcommand{\mc}{\mathcal{C}}
\newcommand{\mm}{\mathcal{M}}
\newcommand{\ml}{\mathcal{L}}

\DeclareMathOperator{\hdim}{\mathrm{dim}_H}

\newcommand{\bfalpha}{\bm{\alpha}}
\newcommand{\bfa}{\bm{\alpha}}

\renewcommand{\P}{\mathbb{P}}
\newcommand{\E}{\mathbb{E}}
\newcommand{\F}{\mathcal{F}}
\newcommand{\cb}{\mathrm{C}\beta \mathrm{E}}

\newcommand{\Cov}{\operatorname{Cov}}
\newcommand{\onef}{\textbf{1}}
\newcommand{\1}[1]{\onef_{\{ #1 \}}}
\newcommand{\Pl}{\mathbf{P}}
\newcommand{\mua}{\mu_{\bfa}}
\newcommand{\Leb}{\mathrm{Leb}}
\newcommand{\Om}{\Omega}

\renewcommand{\th}{\theta}
\newcommand{\dth}{\mathrm{d} \th}
\newcommand{\dmu}{\mathrm{d} \mu}
\newcommand{\dmua}{\dmu_{\bfa}}
\newcommand{\dQa}{\mathrm{d} Q(\bfa)}
\newcommand{\om}{\omega}

\newcommand{\rmd}{\mathrm{d}}

\renewcommand{\a}{\alpha}
\newcommand{\s}{\sigma}
\newcommand{\Dinf}{\D^{\infty}}

\newcommand{\pinf}{+\infty}
\renewcommand{\bar}{\overline}
\renewcommand{\l}{\lambda}
\newcommand{\vphi}{\varphi}

\newcommand{\g}{\gamma}
\renewcommand{\b}{\beta}
\newcommand{\eps}{\varepsilon}
\newcommand{\dl}{\delta}

\renewcommand{\k}{\kappa}
\newcommand{\ds}{\: \mathrm{d}s}
\newcommand{\dt}{\: \mathrm{d}t}

\newcommand{\one}[1]{\mathds{1}_{\left \{ #1 \right \} }}

\newcommand{\ofrac}[2]{o \left ( \frac{#1}{#2} \right )}
\newcommand{\wa}{\widetilde{\a}}
\newcommand{\Dbar}{\overline{\D}}

\begin{document}

\title{Dimension Results for the Spectral Measure of the Circular $\beta$ Ensembles}

\author{\begin{tabular}{ccc}
Tom Alberts\footnote{alberts@math.utah.edu} &  Raoul Normand\footnote{rjn5@nyu.edu} \\
\small{Department of Mathematics} & \small{Courant Institute of Mathematical Sciences}  \\
\small{University of Utah} & \small{New York University} \\
\small{SLC, UT, USA} & \small{New York, NY, USA} \\
\end{tabular}
\\}

\maketitle

\begin{abstract}
We study the dimension properties of the spectral measure of the Circular $\beta$-Ensembles. For $\beta \geq 2$ it it was previously shown by Simon that the spectral measure is almost surely singular continuous with respect to Lebesgue measure on $\bd$ and the dimension of its support is $1 - 2/\beta$. We reprove this result with a combination of probabilistic techniques and the so-called Jitomirskaya-Last inequalities. Our method is simpler in nature and mostly self-contained, with an emphasis on the probabilistic aspects rather than the analytic. We also extend the method to prove a large deviations principle for norms involved in the Jitomirskaya-Last analysis.
\end{abstract}


\section{Introduction}

\subsection{Circular $\b$ ensemble}

The Circular $\beta$-Ensemble $\cb_n$ is the point process of $n$ random points on the unit circle that are distributed according to the law
\begin{align}\label{eqn:circ_beta_density}
\frac{1}{Z_{n, \beta}} \prod_{1 \leq i < j \leq n} \left | e^{i \theta_i} - e^{i \theta_j} \right|^{\beta} \, \frac{\dth_1}{2 \pi} \ldots \frac{\dth_n}{2 \pi}.
\end{align}
The partition function $Z_{n, \beta}$ normalizes the density to be a probability measure and is known explicitly \cite{Good, Wilson}. The circular $\beta$-ensemble was introduced by Dyson \cite{Dyson} who also observed that its density is the Gibbs measure for $n$ identical charged particles confined to the unit circle and interacting via the two-dimensional Coulombic repulsion. For this reason the point process is also sometimes called a log-gas. It is also known that \eqref{eqn:circ_beta_density} is the stationary density for the Dyson Brownian motion on $\bd$, with the strength of the interaction term determined by $\beta$.

That the points live on the unit circle suggests that they should be the eigenvalues of a random unitary matrix, and this is well known in the $\beta = 2$ case (the circular unitary ensemble). For general $\beta$, the corresponding random matrices were constructed in \cite{KillipNenciu:Cbeta} using the so-called \textit{CMV representation} \cite{CMV, KillipNenciu:CMV, Simon1}. CMV is a very general yet efficient matrix representation of a unitary operator on a Hilbert space, analogous to the representation of a self-adjoint operator by a Jacobi matrix. The CMV representation is an infinite matrix that can be seen as an operator on $\ell^2(\N)$, parameterized by an infinite sequence $\bfa = (\alpha_0, \alpha_1, \ldots)$ taking values in $\bar{\D}$, the closure of the unit disk. These coefficients are now commonly referred to as the \textbf{Verblunsky coefficients}, and the correponding matrix denoted $\mc(\bfa)$. To construct a random unitary matrix whose eigenvalue distribution is the Circular $\beta$-Ensemble for general $\b > 0$, it is therefore enough to find an appropriate sequence of \textit{random} Verblunsky coefficients, and construct the corresponding CMV matrices. This is the approach successfully used in \cite{KillipNenciu:Cbeta}, though we will use the slightly different choice of Verblunskys made in \cite{KillipRyckman}. They prove that the appropriate choice is to take the $\alpha_j$ independent with law given by
\begin{align}\label{eqn:CBeta_density}
\alpha_j \sim e^{2 \pi i U} \sqrt{\operatorname{Beta} \left(1, \frac{\beta}{2} (j+1) \right)},
\end{align}
where $U$ is a Uniform random variable on $(0,1)$ that is independent of the Beta variable, and all are independent across $j$. Another way to put it is that the pdf of $\alpha_j$ with respect to the Lebesgue measure on $\D$ is
\[
f_j(z) = \frac{\beta(j+1)}{2 \pi}(1-|z|^2)^{\beta(j+1)/2 - 1}.
\]
Note that these variables are independent under all rotations and hence have mean zero, while the second moment of their radial part decays like
\begin{equation} \label{eq:2ndmoment0}
\E[|\alpha_k|^2] = \frac{2}{\beta(k+1) + 2} \sim \frac{2}{\beta k}.
\end{equation}
It is shown in \cite{KillipNenciu:Cbeta, KillipRyckman} that if $U_0$ is another independent Uniform$(0,1)$ random variable then the CMV matrix
\begin{align}\label{eqn:finite_n_CMV}
\mc_n := \mc(\alpha_0, \alpha_1, \dots, \alpha_{n-2}, e^{2 \pi i U_0}, 0, 0, 0, \dots)
\end{align}
has \eqref{eqn:circ_beta_density} as its eigenvalue density. We describe the CMV matrices more precisely in Section \ref{sec:CMV}, but for now one can think of them as infinite matrices. It will follow that $\mc_n$ has zeros on the rows and columns of order greater than $n$ and its $n \times n$ minor is a matrix with eigenvalue distribution given by \eqref{eqn:CBeta_density}. Moreover, these operators are ``nested'' within each other as $n$ increases. The operator $\mc_{n+1}$ is obtained from $\mc_{n}$ by simply shifting the $e^{2 \pi i U_0}$ entry up by one and inserting $\alpha_{n-1}$ into the now empty slot. In this sense $\mc_{n}$ is a minor of $\mc_{n+1}$, and in fact all $\mc_{n}$ are minors of the infinite matrix $\mc(\bfa)$. This feature is present in the CMV matrices used in \cite{KillipRyckman} but not those introduced in \cite{KillipNenciu:Cbeta}, which is the reason behind our choice of these Verblunskys. 

As $\bfa$ is a random family of Verblunsky coefficients in $\D$, the operator $\mc(\bfa)$ is a random operator on $\ell^2(\N)$, and it can be shown that it has $(1,0,0,\dots)$ as a cyclic vector. It therefore has a spectral measure $\mua$ on $\bD$, which encodes information about both the eigenvalues and eigenvectors of the operator. As the sequence $\bfa$ is random, it therefore gives rise to a random measure in the usual sense of the term, see \cite{KallenbergRM}. This is what we shall call the \textbf{spectral measure of the Circular $\b$-Ensemble}. Note that the choice of Verblunksy coefficients from \cite{KillipNenciu:Cbeta} would have a different spectral measure, despite having the same eigenvalue distribution. This is because the spectral weights attached to the eigenvalues are different.

The spectral measure is also intimately connected to the \textbf{Liouville quantum gravity measure} on $\bd$, as was recently proved in \cite{CN:CBEisGMC}. Liouville quantum gravity has been of intense interest in probability theory and statistical mechanics in recent years for its relation to Schramm-Loewner evolution, conformal field theory, and other parts of the broader field of random conformal geometry (see \cite{FB:freezing, AJKS:welding, Webb:RUM, MS:QLE, Remy:Fy-Bouch} among others). Perhaps its simplest description is via Gaussian Multiplicative Chaos, where one begins with a mean zero, log-correlated Gaussian field $X$ on $\bd$ with $\Cov(X(e^{i \theta}), X(e^{i \phi})) = -2 \log |e^{i \theta} - e^{i \phi}|$ and defines the GMC measure as the following limit, in the topology of weak convergence of measures: 
\[
 e^{\frac{\gamma}{2} X(e^{i \theta})}\;  \dth := \lim_{\epsilon \to 0} e^{\frac{\gamma}{2} X_{\epsilon}(e^{i \theta}) - \frac{\gamma^2}{8} E[X_{\epsilon}(e^{i \theta})^2]}\; \dth,
\]
where $X_{\epsilon}$ is an appropriately smoothed out version of $X$ that converges to $X$ in an appropriate sense as $\epsilon \to 0$. The limiting measure is shown to exist and be non-zero for all $\gamma \in [0,2)$, although the notation for it is somewhat abusive since the limiting measure is almost surely singular with respect to Lebesgue so long as $\gamma > 0$. Most notably, the limiting measure is not a probability measure on $\bd$ and its total mass is a random variable (see \cite{FB:freezing, CN:CBEisGMC, Remy:Fy-Bouch} for recent calculations of its density), in contrast with the spectral measure of the Circular $\b$-Ensemble which is a random probability measure for all $\beta > 0$. In particular, for the GMC the measure and the total mass go to zero as $\gamma \uparrow 2$, while the spectral measure of Circular $\b$ is singular continuous for $\beta \geq 2$ but undergoes a phase transition to an atomic measure as $\beta \downarrow 2$. To observe a similar phase transition for the GMC requires subtle renormalizations as $\gamma \uparrow 2$ (see \cite{APS:critical_Liouville}) which are not needed for Circular $\beta$, thanks to the theory of orthogonal polynomials on the unit circle. 

In the subcritical phase $\gamma < 2$ and the critical phase $\gamma = 2$ ($\beta \geq 2$) Chhaibi and Najnudel \cite{CN:CBEisGMC} recently proved an exact equality between GMC and the spectral measure of the Circular-$\beta$ Ensemble. More precisely, they determine an explicit real-valued function of the Verblunsky coefficients whose product with the spectral measure is \textit{exactly} the GMC measure on $\bd$, with the relationship $\gamma^2 = 8/\beta$. Note that \cite{CN:CBEisGMC} is actually that $\gamma^2 = 2/\beta$, but this is because they use $e^{\gamma X}$ as their normalization rather than our $e^{(\gamma/2) X}$ (with the same $2$ in the covariance kernel of the $X$). One of our reasons for this is that it maintains the phase transition at $\gamma = 2$ as in \cite{DS:LQG_KPZ}, but we also note that our normalization is also what is used in \cite{Remy:Fy-Bouch}. Also note that the result of \cite{CN:CBEisGMC} is somewhat similar to the result of Webb \cite{Webb:RUM}, who proves that (normalized) powers of the characteristic polynomial of random unitary matrices converge to GMC on $\bd$. The relation $\gamma^2 = 8/\beta$ agrees with the phase transition occurring at both $\gamma = 2$ and $\beta = 2$. Furthermore, the relation $\gamma = \sqrt{\kappa}$ between GMC and SLE processes \cite{DS:LQG_KPZ} leads to $\beta = 8/\kappa$, which was earlier conjectured by Cardy \cite{Cardy:CS} in his relation of Dyson's Brownian motion to multiple SLE. Following the analysis in \cite{DS:LQG_KPZ, HMP:thick_points, Berestycki:elementary_GMC} the GMC measure is supported on a set of Hausdorff dimension $1 - \gamma^2/4$, which by the identification above is equal to $1-2/\beta$. One of our main results is a new proof of this dimension formula using the tools of orthogonal polynomials on the unit circle and ideas from probability theory, which we describe next. 

\subsection{Results}

In this paper we focus on the almost sure dimension properties of the spectral measure of the Circular $\b$-Ensemble. We reprove a result of Simon \cite[Theorem 12.7.7]{Simon2} on the almost sure Hausdorff dimension of the support of the random measure by combining the technique of \textbf{size-biasing} with the \textbf{Jitomirskaya-Last inequalities} for computing local dimensions via orthogonal polynomials. Recently, Simon's result was invoked in \cite[Corollary 2.4]{CN:CBEisGMC} to give an alternative proof for the dimension of the support of the Gaussian multiplicative chaos via the direct link they establish to the spectral measure of the Circular $\beta$-ensemble. As we explain shortly, in the background of Simon's result is an appeal to the abstract but \textit{deterministic} theory of Kotani-Ushiroya \cite[Theorem 10.5.34]{Simon2}, which is quite involved. Our methods bypass the Kotani-Ushiroya theory and instead use probabilistic ideas to prove the same result as \cite[Theorem 12.7.7]{Simon2}, in the case of certain random Verblunsky coefficients and with minor modifications to the assumptions. We also prove a large deviation principle that is relevant to studying the fine dimensional properties of the measure \cite{Virag:ICM_operator}. Our results do not rely on the exact form \eqref{eqn:CBeta_density} for the law of the Verblunsky coefficients. We shall merely use the independence, the rotational symmetry of their laws, and the asymptotics \eqref{eq:2ndmoment0}. To be precise, let us denote by $Q$ the law of the family $(\a_0,\a_1,\dots)$ on $\D^{\infty}$, endowed with the product $\s$-algebra. This gives rise to a random measure $\mu_{\bfa}$ on $\bd$. We make the following assumptions, which shall become relevant at different (but clearly stated) parts of the text.

\begin{assumption} \label{as:indrot}
The Verblunsky coefficients $\bfalpha = (\a_0, \a_1, \ldots)$ take values in $\D$, are independent, and each have a radially symmetric distribution. In other words, $Q$ is a product measure on $\D^{\infty}$, and each marginal distribution is invariant under arbitrary rotations of $\D$.
\end{assumption}

\begin{assumption} \label{as:2ndmoment}
The second moments of the Verblunsky coefficients decay as
\begin{equation} \label{eq:2ndmoment}
\E[|\alpha_k|^2] \sim \frac{2}{\beta k}.
\end{equation}
\end{assumption}

These are the main assumptions that we will need. The factor $\b$ controls the dimension of the measure, as Theorem \ref{th:exactdim} below shows. Note that these two assumptions are clearly satisfied by the Verblunsky coefficients of the Circular $\beta$-Ensemble, which in fact satisfy $\E[|\a_k|^m] \sim C_m k^{-m/2}$ as $k \to \infty$ for all positive integers $m$. In general we do not require such strong asymptotics for the higher moments, but our two main results will also require some additional assumptions on higher moments that rule out the $\a_k$ having too much mass near the unit circle. For instance, distributions for $|\a_k|$ of the type
\[
\left ( 1 - \frac{1}{k^3} \right ) \dl_{k^{-1/2}} + \frac{1}{k^3} \dl_{1 - 1/k}
\]
would not be allowed.

\subsubsection{Exact dimension of the measure}

For our first result recall the following definition of the exact Hausdorff dimension of a measure $\mu$.

\begin{definition}
A measure $\mu$ on a measure-metric space has \textit{exact Hausdorff dimension $\k$} if it is supported on a set of Hausdorff dimension $\k$, and gives zero mass to all sets of dimension strictly smaller than $\k$.
\end{definition}

\begin{theorem} \label{th:exactdim}
Assume that the Verblunsky coefficients $\alpha_k$ satisfy Assumptions \ref{as:indrot} and \ref{as:2ndmoment} for some $\beta \geq 2$, and that the third moments also satisfy
\begin{equation} \label{eq:3rdmoment}
\\E[|\alpha_k|^3] = o(k^{-1}).
\end{equation}
Then with $Q$ probability one, the exact Hausdorff dimension of $\mu_{\bfa}$ is $1 - 2/\beta$.
\end{theorem}

A similar result is proved by Simon in \cite{Simon2}, albeit with slightly different assumptions (see also \cite{Breuer} for similar results for random tridiagonal matrices). Simon also assumes the independence of the Verblunskies and the same decay of the second moment, but instead of rotational invariance, he only assumes the weaker condition that $\E[\alpha_k] = \E[\alpha_k^2] = 0$ for each $k$. However, he also makes the extra assumptions that
\begin{equation} \label{eq:Simon_assumptions}
\sup_{n,\om} |\alpha_n| < 1, \quad \sup_{n,\om} \sqrt{n} |\alpha_n| < \infty,
\end{equation}
which we do not require and is not satisfied for the $\b$-ensemble. Simon also considers the so-called Aleksandrov measures $\mu_{\l \bfa}$, where $\l \in \bd$ and $\l \bfa$ means the rotated coefficients $(\l \a_0, \l \a_1, \dots)$, and his result is the following:

\begin{theorem}[{\cite[Theorem 12.7.7]{Simon2}}] \label{thm:Simon_dim}
With $\dQa \, \rmd \lambda/2\pi$ probability one the measure $\mu_{\l \bfa}$ has exact Hausdorff dimension $1-2/\b$.
\end{theorem}

In one sense this is a more general result as no rotational independence is required. On the other hand, when the $\a_k$ do have a rotationally symmetric distribution then all the $\mu_{\l \bfa}$ have the same distribution, and we therefore recover our Theorem \ref{th:exactdim}. However, this comes at the price of Assumptions \eqref{eq:Simon_assumptions} above, which are not satisfied in many situations, most notably for the spectral measure of the $\cb$. A truncation argument similar to what we apply in Section \ref{sec:bounded_Ver} would overcome this difficulty, but as we explain next the main strength of our proof is that it avoids the underlying \textit{Kotani theory} on which Simon's proof is based. This has the advantage of making the proof more self-contained, and highlights the usefulness of the underlying probabilistic ideas.

Both our proofs and Simon's are based on an analysis of the \textit{local dimension} of the measure $\mu_{\bfalpha}$ at different points. Loosely speaking, the local dimension of a measure $\nu$ on $\bd$ at a point $e^{i \theta}$ is the exponent $s_0 = s_0(\nu, \theta) \in [0,1]$ such that
\[
\nu(\theta - \epsilon, \theta + \epsilon) \approx \epsilon^{s_0}
\]
as $\epsilon \to 0$ (see Section \ref{sec:dimension_theory} for the precise definition). The dimension theory of Rogers \& Taylor \cite{RogersTaylor:ActaI, RogersTaylor:ActaII} implies that if the local dimension $\theta \mapsto s_0(\nu, \theta)$ is almost surely constant under the measure $\nu$, then $\nu$ has exact Hausdorff dimension equal to this constant. Computing $s_0(\nu,\theta)$ at a particular point usually requires working with some approximating sequence $\nu_{\delta}$ of smoothed out versions of the measure, which is often difficult because it requires precise estimates on the relationship between the scale of the approximation $\delta$ and the width of the interval $\epsilon$ in the quantity $\nu_{\delta}(\theta - \epsilon, \theta + \epsilon)$. For probability measures on the real line, a particularly nice method of approximation using the associated orthogonal polynomials was developed by Last \cite{Last:SingContSpectra} and Jitomirskaya-Last \cite{JitoLast:Acta}. Later, Simon \cite{Simon2} translated their results to probability measure on $\bd$ and OPUC, which forms the basis of his proof of Theorem \ref{thm:Simon_dim}. In fact, this is where we learned of the Jitomirskaya-Last technique. In short, it says that if $\varphi_n : \bd \to \C, n \geq 0,$ are the orthonormal polynomials in the Hilbert space $L^2(\bd, \rmd \nu)$ and $\psi_n : \bd \to \C, n \geq 0,$ are the associated \textbf{second kind polynomials} (see Section \ref{sec:second_kind} for the definition) then 
\[
s > s_0(\nu, \theta) \iff \liminf_{n \to \infty} \frac{\|\varphi_{\cdot}(e^{i \theta})\|_n^{2-s}}{\|\psi_{\cdot}(e^{i \theta})\|_n^s} = \infty,
\]
where $\|a_{\cdot}\|_n^2$ is the squared norm on sequences $a_{\cdot} = \{a_n \}_{n = 0}^{\infty}$ in $\C$ defined by
\[
\|a_{\cdot}\|_n^2 = \sum_{k=0}^n |a_k|^2.
\]
When applied to the orthogonal polynomials we refer to these as the Jitomirskaya-Last norms. Our analysis of the norms is somewhat similar to Simon's but the main difference in our proofs comes in the way we choose $\theta$. Simon appeals to a powerful theorem of Kotani-Ushiroya \cite[Theorem 10.5.34]{Simon2} that relates the growth of norms to the existence of points where the measure has a particular local dimension. The advantage of the Kotani-Ushiroya result is that it is purely deterministic, but the underlying theory is difficult and much broader than what is needed for random Verblunsky coefficients. Our approach is briefer and entirely self-contained. It relies only on classical probabilistic ideas: martingale arguments, laws of large numbers, and coupling techniques, and gives a full description of which points $\theta$ are the appropriate ones to look at. We analyze the Jitomirskaya-Last norms by considering them as random variables under the joint measure $dQ(\bfalpha) \, d \mu_{\bfalpha}(\theta)$, which is a skew product on $\D^{\infty} \times \bd$ with the Verblunskies chosen first and then the point $\theta$ sampled according to the measure determined by the Verblunskies. We give an alternative description of this measure using the \textbf{Bernstein-Szeg\"{o} approximation}, which says that for each fixed $n$ the quantity $|\varphi_n(e^{i \theta})|^{-2} \, \dth/2\pi$ is a probability measure on $\bd$, and that
\[
|\varphi_n(e^{i \theta})|^{-2} \, \frac{\dth}{2 \pi} \xrightarrow{n \to \infty} \dmua(\theta)
\]
in the sense of weak convergence of measures. The Bernstein-Szeg\"{o} approximation allows us to rewrite the joint measure $dQ(\bfalpha) d \mu_{\bfalpha}(\theta)$ in terms of the marginal measure of the point $\theta$ and the conditional measure of the Verblunsky coefficients, which turns out to be a much simpler description of the random pair. The conditional measure turns out to be a size-biasing of the original Verblunsky coefficients with nice properties. In particular there is a natural coupling between the original Verblunskies and the size-biased ones that in turn allows us to easily compute all moments of the size-biased Verblunskies. A similar size-biasing was previously used in \cite{BNR:CJE}, although in a slightly different form and for different purposes. In our work the size-biased Verblunskies are used to compute a ``strong law of large numbers'' for the growth of the Jitormirskaya-Last norms at typical points of the measure $\mu_{\bfalpha}$, which is then combined with the Rogers-Taylor theory to produce the dimension results.

We point out that for probabilists, the combination of Rogers-Taylor and Jitomirskaya-Last dimension theories is potentially useful in other applications. A now relatively standard method of computing Hausdorff dimensions of random sets is via the so-called \textit{second moment method}. Roughly speaking this requires estimates on the expected value and variance of the number of balls in an $\epsilon$-net of the ambient space that are required to cover the random fractal. The required estimates are on the power law blowup of these quantities as $\epsilon \to 0$. The expected number is typically easy to deal with since it only requires  the probability that a fixed point is within a distance $\epsilon$ of the random set, but the variance is more complicated since this requires estimates on probabilities of two fixed points both being within distance $\epsilon$ of the random set. This type of two point estimate requires getting a handle on often complicated correlations and can quickly become messy. In contrast, the Jitomirskaya-Last dimension theory requires two one-point estimates rather than one two-point estimate, and while these two one-point quantities are related the correlation appears to be simpler to deal with. The necessary estimate is handled in Section \ref{sec:Q0_asymptotics}. The downside of course is that Jitomirskaya-Last requires there to be a measure on the random set and some information about the associated orthogonal polynomials, which may not always be available. 

\subsubsection{Large Deviations for the Jitomirskaya-Last Norm}

Our second main result is a large deviations principle (LDP) for the growth of the Jitomirskaya-Last norms. Given that the dimension result of Theorem \ref{th:exactdim} is based on a strong law for the growth of these norms, an LDP is a very natural extension that should have applications in analyzing the subsets of $\bd$ where the measure has atypical local dimension.

\begin{theorem}\label{thm:norm_LD}
Assume the Verblunsky coefficients $\alpha_k$ satisfy Assumptions \ref{as:indrot} and \ref{as:2ndmoment} and that, for some $\epsilon > 0$ and all $\kappa > 0$,
\begin{equation} \label{eq:asLDP}
\E[|\alpha_k|^3] = O(k^{-1-\epsilon}), \quad \limsup_{k \to \infty} \E[(1-|\alpha_k|)^{-\kappa}] < \infty.
\end{equation}
Then under the measure $\rmd \Pl(\bfa,\th) = \rmd Q(\bfalpha) \, \dth / 2 \pi$, the sequence of random variables
\[
\frac{\log \|\varphi_{\cdot}(e^{i \theta})\|_n^2}{\log n}
\]
satisfies a large deviations principle with speed $\log n$ and rate function
\[
I(x) =
\begin{cases}
\frac{\beta}{8} \left( x - 1 - \frac{2}{\beta} \right)^2, & x \geq 0, \\
\infty, & x < 0.
\end{cases}
\]
Moreover, under the measure $\rmd \P(\bfa,\th) = \rmd Q(\bfalpha) \, \rmd \mu_{\bfalpha}(\theta)$, the same sequence satisfies the LDP with speed $\log n$ and rate function
\[
J(x) =
\begin{cases}
\frac{\beta}{8} \left( x - 1 + \frac{2}{\beta} \right)^2, & x \geq 0, \\
\infty, & x < 0.
\end{cases}
\]
\end{theorem}

In particular, these moment assumptions are satisfied by the Verblunsky coefficients of the Circular $\beta$-Ensemble. Indeed, by \eqref{eqn:CBeta_density}, $|\alpha_k|^2$ has a Beta$(1, \beta (k+1)/2)$ distribution, and hence $\E[|\alpha_k|^3] = O(k^{-3/2})$. Moreover, this also implies that $1 - |\alpha_k|^2$ has a Beta$(\beta(k+1)/2, 1)$ distribution, and hence
\[
\E[(1 - |\alpha_k|)^{-\kappa}] \leq C \E[(1 - |\alpha_k|^2)^{-\kappa}] = C' \int_0^1 x^{-\kappa} x^{\beta(k+1)/2-1} \, dx = \frac{C''}{\beta(k+1)/2 - \kappa},
\]
for some constants $C',C''$, provided that $\kappa < \beta(k+1)/2$. Note however that the assumptions of Theorem \ref{thm:norm_LD} means that the quadratic nature of the rate function is not particular to the Circular-$\beta$ ensemble and is therefore universal for a wide class of Verblunsky coefficients, even though the spectral measure itself is in bijection with the Verblunskies. The quadratic rate function is of course a manifestation of the underlying Gaussian field behind the orthogonal polynomials, as discovered in [CITE]. 

The jump in the rate functions at $x = 0$ is because the range $x < 0$ would correspond to a shrinking of $\|\varphi_{\cdot}(1)\|_n^2$ as $n \to \infty$, which is clearly impossible since the norm is the sum of positive terms. That the rate function $J$ has a zero at $x = 1 - 2/\beta$ suggests that $\log \|\varphi_{\cdot}(1)\|_n^2 \sim (1 - 2/\beta) \log n$ almost surely under $\P$, which is a key part of the proof of Theorem \ref{th:exactdim}. The large deviations principle can be seen as a refinement of this result, and its proof also contains more detailed information. In particular, in the proof of Theorem \ref{th:exactdim}, we only need to consider the almost sure asymptotics of $\log |\varphi_n(e^{i \theta})|^{2}/\log n$ which are in turn enough to control the asymptotics of $\log \|\varphi_{\cdot}(e^{i \theta})\|_n^2/ \log n$. But in the large deviations regime, we need to control the behavior of the entire process $k \mapsto \log |\varphi_k(e^{i \theta})|^2$ in order to control the behavior of $\log \|\varphi_{\cdot}(e^{i \theta})\|_n^2$. We do so by proving a process level large deviations principle for $k \mapsto \log |\varphi_k(e^{i \theta})|^{-2}$, suitably rescaled and under an appropriate change of the time scale, which takes up the bulk of our proof of Theorem \ref{thm:norm_LD} in Section \ref{sec:LDP}. The process level LDP that we prove is an extension of Mogulskii's theorem for sums of iid random variables, although we need to relax the assumption of identical distribution in order to deal with the decaying Verblunsky coefficients. The process level LDP (Theorem \ref{thm:sum_LDP}) is useful in its own right since it explains the behavior of the Verblunsky coefficients when $\log |\varphi_n(e^{i \theta})|^{-2}$ takes on an atypical value, i.e. whenever the measure has some unusual behavior around $\th$. Once the process level LDP is proven, converting it into the LDP for the norm is a relatively straightforward application of the contraction principle.   

We point out that large deviations in the context of OPUC have become quite popular in recent years \cite{GNR:LD_sum_rule, GNR:LD_sum_rule_circle, GNR:LD_sum_rule_SMM, BSZ:LD_sum_rule, BSZ:Lukic}, although more towards sum rules rather than dimension theory. In particular, \cite{GNR:LD_sum_rule} uses large deviations techniques to give an alternative proof of the Szeg\"{o}-Verblunsky sum rule for OPUC, with extensions of this idea in the remaining papers. We will not state the sum rule here since it is concerned with the absolutely continuous part of the spectral measure, of which there is none for the types of Verblunsky coefficients that we are considering.    

\subsection{Outline}

The outline of this paper is as follows. Section \ref{sec:OPUC} is background on the OPUC theory and dimension theory that we will use throughout. In Section \ref{sec:RIVB}, we explain how the Bernstein-Szeg\"{o} approximation of measures allows us to study the typical points of the spectral measure via a size-biasing of their Verblunsky coefficients. We also explain the coupling procedure that allows us to transform the original Verblunsky coefficients into their size-biased versions, and how this allows us to compute moments of the size-biased quantities. In Section \ref{sec:local_dim} we use this coupling to prove Theorem \ref{th:exactdim} on the exact dimension of the spectral measure of the Circular-$\beta$ ensemble and other measures coming from rotationally invariant Verblunsky coefficients. Finally, in Section \ref{sec:LDP}, we prove the large deviations principle for the Jitomirskaya-Last norm, using an extension of results for process level large deviations. 

\vspace{5mm}
\noindent
\textbf{Acknowledgements:} We thank B\'{a}lint Vir\'{a}g for many helpful discussions, especially for suggesting the proof of Section \ref{sec:secondkind}.  Alberts was partially supported by NSF grants DMS-1715680, DMS-1811087, and Simons Collaboration Grant 351687.

\section{Background for OPUC Theory} \label{sec:OPUC}

In this section we review the deterministic theory behind the theory of Orthogonal Polynomials on the Unit Circle (OPUC), which will be our main technique to study the measure $\mu_{\bfa}$. The exposition and notation of this section very closely follows that in \cite{Simon1, Simon2}, where the interested reader can find a much more detailed treatment. We emphasize that almost everything in this section is deterministic in nature; the random case is considered in the following section.

\subsection{CMV matrices} \label{sec:CMV}

Originally discovered in \cite{CMV}, the CMV matrices are canonical representatives of unitary operators acting on a separable Hilbert space, much like tri-diagonal matrices are canonical representatives for self-adjoint operators. CMV matrices turn out to be five-diagonal but it is known \cite{Simon1} that they are the sparsest possible matrix representations of a unitary operator, in the sense that no fewer number of diagonals can represent all unitary operators.

The CMV matrices are defined as operators on the complex sequence space $\ell^2(\N)$, where $\N=\{0,1,2,\dots\}$, with the inner product
\begin{align*}
\langle ( a_n)_{n=0}^{\infty}, ( b_n )_{n=0}^{\infty} \rangle = \sum_{n=0}^{\infty} \overline{a_n} b_n.
\end{align*}
Note that this inner product is antilinear in the first factor, which is the convention we follow throughout. The input to forming a CMV matrix is an infinite sequence of complex numbers $\{ \alpha_n \}_{n=0}^{\infty}$ taking values in the closure $\Dbar$ of the unit disk $\D$. For each such sequence we make the following definition.

\begin{definition}\label{defn:CMV}
Let $\bfa = \{ \alpha_j, j \geq 0 \}$ be a sequence taking values in $\Dbar$. Let
\begin{align}\label{defn:rho}
\rho_j = \sqrt{1 - |\alpha_j|^2}
\end{align}
and define the $2 \times 2$ matrices $\Theta_j$ by
\begin{align}\label{defn:Theta_matrices}
\Theta_j = \left(
             \begin{array}{cc}
               \overline{\alpha}_j & \rho_j \\
               \rho_j & -\alpha_j \\
             \end{array}
           \right).
\end{align}
Let $\mm$ and $\ml$ be the infinite matrices
\begin{align*}
\mm = \left(
  \begin{array}{cccc}
    1 &  &  &  \\
     & \Theta_1 &  &  \\
     &  & \Theta_3 &  \\
     &  &  & \ddots \\
  \end{array}
\right),
\quad
\ml = \left(
         \begin{array}{cccc}
             \Theta_0 &  &  &  \\
             & \Theta_2 &  &  \\
             &  & \Theta_4 &  \\
             &  &  & \ddots \\
         \end{array}
     \right),
\end{align*}
where the $1$ in $\mm$ is just the number $1$, not an identity matrix, and the blank entries are all zero. The $\Theta_j$ matrices are always aligned so that the $\alpha_j$ terms lie on the diagonal. Then the \textit{CMV matrix} associated to the sequence $\bfa$ is the infinite matrix (or operator) on $\ell^2(\N)$ defined by
\begin{align*}
\mc(\bfa) = \mc(\alpha_0, \alpha_1, \alpha_2, \ldots) := \ml \mm.
\end{align*}
\end{definition}

It is straightforward to see that the CMV matrices are unitary and it is not much more difficult to compute that they are at most five-diagonal. The complex numbers $( \alpha_n, n \geq 0 )$, historically have more than one name in the literature but are now commonly referred to as the \textbf{Verblunsky coefficients}. We also use this terminology.

\subsection{Spectral Measure}

A key tool in the study of unitary operators is their \textit{spectral measure}, based on the spectral theorem \cite{Simon1} which we now briefly recall. Assume that $U : \Hil \to \Hil$ is a unitary operator on a separable Hilbert space $\Hil$, and further assume that there is a $\zeta \in \Hil$ such that the finite linear span of $\{ U^i \zeta : i \in \Z \}$ is dense in $\Hil$. Such a $\zeta$ is called a \textit{cyclic vector} for $U$. The spectral theorem asserts that there is a probability measure $\mu$ on $\bd$ and a linear mapping $V : \Hil \to L^2(\bd, \dmu)$ such that
\begin{enumerate}[(i)]
\item $V(Ux) = zV(x)$, that is, $U$ on the image side is just multiplication by the \textit{function} $z$;
\item $V$ preserves the inner product, i.e. $\langle Vx, Vy \rangle_{L^2(\bd, \dmu)} = \langle x, y \rangle_{\Hil}$, and hence $V$ is an isometry;
\item $V(\zeta) = 1$, the constant function on $\partial \D$.
\end{enumerate}
This $\mu$ is called the \textbf{spectral measure for the pair} $(U, \zeta)$. Since the mapping $V$ preserves the inner product we see that applying $U$ in $\Hil$ is equivalent to multiplying by $z$ in $L^2(\bd, \dmu)$. Hence studying $U$ is entirely reduced to studying properties of the Hilbert space $L^2(\bd, \dmu)$, and in that sense the spectral measure $\mu$ encodes all the information about $U$. 

For the CMV matrices, if the Verblunsky coefficients $\alpha_n$ all satisfy $|\alpha_n| < 1$ then it can be shown (see the next section) that $\delta_0 = (1,0,0, \ldots)$ is a cyclic vector for $\mc(\bfalpha)$ as an operator on $\ell^2(\N)$. To emphasize the dependence on the Verblunsky coefficients we will write $\mu_{\bfalpha}$ for the spectral measure of the pair $(\mc(\bfalpha), \delta_0)$. The mapping $\bfalpha \mapsto \mu_{\bfalpha}$ is measurable from $\D^{\infty}$ to the space of probability measures on $\bd$; which follows from the Bernstein-Szeg\H{o} approximation (Proposition \ref{prop:BS}).

\subsection{Orthogonal Polynomials}

One of the main tools used to study the spectral measure is the associated \textit{orthogonal polynomials}. Given a probability measure $\mu$ on $\bd$, let $L^2(\bd, \dmu)$ be the space of complex-valued, square-integrable functions with respect to the inner product
\begin{equation}
\langle f, g \rangle_{L^2(\bd, \dmu)} = \int_{\bd} \overline{f(e^{i \theta})} g(e^{i \theta}) \; \dmu(\theta),
\end{equation}
which we note is antilinear in the first factor. If $\mu$ is supported on an infinite number of points, then the functions $\{1, z, z^2, \dots\}$ are linearly independent in $L^2(\bd, \dmu)$, and hence we can use the Gram-Schmidt procedure to construct the sequence of orthogonal polynomials
\begin{equation} \label{eq:Phin}
\Phi_n(z) = \Phi_n(z; \dmu). 
\end{equation}
That is, $\Phi_n(z; \dmu)$ is the $L^2(\bd, \dmu)$ projection of $z^n$ on $\{1, z, \ldots, z^{n-1} \}^{\perp}$. If $\mu$ is supported on exactly $n$ distinct points, then the Gram-Schmidt procedure can still be applied but stops with $\Phi_{n-1}$. However, we will assume in the following that this not the case. In particular, our assumptions will ensure that the measures $\mu$ that we consider do have infinite support.

Now, note that $\Phi_0(z) = 1$ and the subsequent $\Phi_n$ are monic by definition. One of their key properties is that they satisfy the so-called \textit{Szeg\H{o} recurrence}:
\begin{equation} \label{eqn:szego_recurrence}
\Phi_{n+1}(z) = z \Phi_n(z) - \overline{\alpha}_n \Phi_n^*(z),
\end{equation}
where the $*$ operation takes the coefficients of an $n$th degree polynomial, reverses them and then conjugates them, i.e.
\begin{equation}
Q_n(z) = \sum_{j=0}^n q_j z^j \implies Q_n^*(z) = \sum_{j=0}^n \overline{q}_{n-j} z^j.
\end{equation}
The $\alpha_n$ terms that appear in \eqref{eqn:szego_recurrence} are precisely the Verblunsky coefficients, in the sense that $\mu$ is the spectral measure for the CMV matrix $\mc(\alpha_0, \alpha_1, \ldots)$ and the cyclic vector $\delta_0$, see Theorem 4.2.8 in \cite{Simon1}. Hence the theory of orthogonal polynomials provides a way to realize a  probability measure on $\bd$ as the spectral measure of a CMV matrix. It also provides a way to compute the CMV representation of a general unitary operator, by first going through the spectral measure and then pulling out the Verblunsky coefficients.

Another way to reconstruct the probability measure from the orthogonal polynomials is via the \textit{Bernstein-Szeg\H{o} approximation}. Let $\varphi_n$ be the normalized orthogonal polynomials defined by
\[
\varphi_n := \frac{\Phi_n}{\|\Phi_n\|},
\]
where the bars in the denominator indicate the $L^2(\bd, \dmu)$ norm. Since we will use the Bernstein-Szeg\H{o} approximation repeatedly we record it as a proposition.

\begin{proposition}[Bernstein-Szeg\H{o} approximation] \label{prop:BS}
For any sequence $\bfalpha \in \D^{\infty}$ the measures
\begin{equation} \label{eqn:berstein_szego_approx}
|\varphi_n(e^{i \theta})|^{-2} \, \frac{\dth}{2\pi}
\end{equation}
are probability measures on $\bd$, and they converge weakly to the spectral measure $\mu$ as $n \to \infty$.
\end{proposition}

Although we will not use this fact, it is worth noting that the measures \eqref{eqn:berstein_szego_approx} also turn out to be the spectral measures for the CMV matrices $\mc(\alpha_0, \alpha_1, \ldots, \alpha_{n-1}, 0,0,0, \dots)$.

%
%

\subsection{Rotations of the Measure} \label{sec:rotation}

For a sequence of Verblunsky coefficients $\bfa$ and the corresponding measure $\mua$, the following result identifies the Verblunsky coefficients of the rotated measure $\mua(\l \cdot)$ for $\l \in \bD$.

\begin{lemma} \label{lemma:rotated_Verblunskies}
For any sequence of Verblunsky $\bfa = \{\a_n\}_{n=0}^{\infty}$ in $\D$, and any $\l \in \bD$, we have
\[
\Phi_n(\l z; \{ \a_n \}) = \l^n \Phi_n(z,\{\l^{n+1} \a_n \}),
\]
for all $z \in \bar{\D}$ and $n \in \N$. In particular, the Verblunsky coefficients of the rotated measure $\mua(\l \cdot)$ are $\{\l^{n+1} \a_n \}$.
\end{lemma}

\begin{proof}
The first formula follows by induction on $n$, by using the Szeg\H{o} recursion \eqref{eqn:szego_recurrence}. Then, the OPUC for the rotated measure are
\[
\Phi_n(z;\mua(\l \cdot)) = \bar{\l}^n \Phi_n(\l z; \mua),
\]
since the right hand side defines a family of monic polynomials which are orthogonal in $L^2(\bD, \mua(\l \cdot))$, as can be seen by a simple change of variable. Putting the pieces together, we thus get
\[
\Phi_n(z;\mua(\l \cdot)) = \bar{\l}^n \Phi_n(\l z;\mua) = \bar{\l}^n \Phi_n(\l z; \{\a_n \}) = \Phi_n(z,\{\l^{n+1} \a_n \}),
\]
which exactly means that the Verblunsky coefficients of $\mua(\l \cdot)$ are $\{\l^{n+1} \a_n \}$.
\end{proof}

\subsection{Second Kind Polynomials and the Transfer Matrix \label{sec:second_kind}}

The second kind polynomials correspond to rotated versions of the Verblunsky coefficients. For Verblunsky coefficients $\bfalpha = \{ \alpha_n, n \geq 0 \}$, the \textit{Aleksandrov measures} are the probability measures on $\bd$ with Verblunsky coefficients $\lambda \bfalpha = \{\lambda \alpha_n, n \geq 0 \}$, i.e. $\mu_{\lambda \bfalpha}$, for $\lambda \in \bd$. It is important to note that the Aleksandrov measures are \textit{not} just a rotation of the original measure, that is $\mu_{\lambda \bfalpha}$ is different from $\mua(\l \cdot)$.

The case $\lambda = -1$ turns out to be very important, and the corresponding polynomials are denoted by $\Psi_n$ and $\psi_n$, i.e.
\begin{equation}\label{eqn:second_kind_polynomials}
\Psi_n(z) = \Phi_n(z;\dmu_{- \bfa}), \quad \psi_n(z) = \vphi_n(z;\dmu_{- \bfa}).
\end{equation}
The $\Psi_n$ are referred to as the \textit{second kind polynomials}. The Szeg\H{o} recursion for $\Psi_n$ becomes
\begin{equation} 
\Psi_{n+1}(z) = z \Psi_n(z) + \overline{\alpha_n} \Psi_n^*(z).
\end{equation}
There is a similar Szeg\H{o} recurrence for $(\Phi_n^*)$ and $(\Psi_n^*)$, and their renormalized versions. If we define
\[
P_n(z) =
\begin{pmatrix}
\psi_n(z) & \vphi_n(z) \\
 -\psi_n^*(z) & \vphi_n^*(z)
\end{pmatrix}
\]
then the recursion for $P_n(z)$ is 
\begin{equation} \label{eqn:szego_recurrence_matrix}
P_{n+1}(z) = A_n Z P_n(z), \quad
A_n =
\rho_n^{-1/2}
\begin{pmatrix}
1 & - \bar{\a}_n \\
- \a_n & 1 \\
\end{pmatrix},
\quad
Z =
\begin{pmatrix}
z & 0 \\
0 & 1
\end{pmatrix}.
\end{equation}
The matrices $A_n$ (or $A_n Z$) are referred to as the transfer matrices. The recurrence formula for $\varphi_n$ is particularly useful for our purposes. We write it using an auxiliary quantity $B_n(z)$ defined as
\begin{equation} 
B_n(z) = \frac{z \vphi_n(z)}{\vphi_n^*(z)}.
\end{equation}
From \eqref{eqn:szego_recurrence_matrix} we can write the Szeg\H{o} recursion for $\varphi_n$ as
\begin{equation} \label{eqn:szego_recurrence_VW}
\varphi_{n+1}(z)  = \rho_n^{-1} (z \varphi_n(z) - \overline{\alpha_n} \varphi_n^*(z)) = \rho_n^{-1} z \varphi_n(z) (1 - \overline{\alpha_n B_n(z)}). 
\end{equation}
We will be mostly interested in the case $|z| = 1$, for which the definition of the $*$ operation implies that $\varphi_n^*(e^{i \theta}) = e^{i n \theta} \overline{\varphi_n(e^{i \theta})}$. Therefore the definition of $B_n$ becomes
\begin{equation}  \label{eqn:Bn_defn}
B_n(e^{i \theta}) = e^{-i (n-1) \theta} \varphi_n(e^{i \theta}) / \overline{\varphi_n(e^{i \theta})} 
\end{equation}
from which one easily sees that $B_n(e^{i\th})$ merely registers the argument of $\vphi_n(e^{i\th})$ and $|B_n(e^{i \theta})| = 1$. More generally, the definition of $\varphi_n^*$ implies that $B_n$ is a finite Blaschke product from which the last fact readily follows. From the recursion \eqref{eqn:szego_recurrence_VW} and the definition \eqref{eqn:Bn_defn} we obtain the recurrence 
\begin{equation} \label{eqn:V_recursion}
B_{n+1}(e^{i\th}) = e^{i\th} B_n(e^{i\th}) \frac{1 - \overline{\alpha_n B_n(e^{i\th})}}{1 - \alpha_n B_n(e^{i\th})}, \quad B_0(e^{i\th}) = e^{i\th}.
\end{equation}
The initial condition is from \eqref{eqn:Bn_defn} and $\varphi_n(e^{i \theta}) = 1$.

Finally, in view of the Berstein-Szeg\H{o} approximation, we shall need to study the sequence $(|\vphi_n(e^{i\th})|^{-2})$. Taking the modulus of both sides of \eqref{eqn:szego_recurrence_VW} and using the definition \eqref{defn:rho} of $\rho_n$ gives
\begin{equation} \label{eqn:phi_psi_mod_recursion}
|\varphi_{n+1}(e^{i \theta})|^{-2} = |\varphi_n(e^{i \theta})|^{-2} \frac{1 - |\alpha_n|^2}{|1 - \alpha_n B_n(e^{i \theta})|^2}.
\end{equation}
This recursion will be very important in our later analysis.

\subsection{Pr\"{u}fer Phases and the Modified Verblunsky Coefficients} \label{sec:modifiedVerb}

The recursion formula \eqref{eqn:phi_psi_mod_recursion} and the Bernstein-Szeg\H{o} approximation demonstrate that the quantities $\alpha_n B_n(e ^{i \theta})$ will play a central role in our study. Due to rotation invariance of the laws of our Verblunsky coefficients, it is usually enough to study $\varphi_n$ at a single point, for which we use $z = 1$. Following \cite[Lemma 2.1]{KillipStoiciu} (and similar constructions in \cite{BNR:CJE, CMN:max-cBeta}) we therefore define the modified Verblunsky coefficients by
\begin{equation} \label{eq:modified_Verblunsky}
\gamma_k = \alpha_k B_k(1), \quad k \geq 0.
\end{equation}
This definition sets up a bijection between the modified Verblunskies and the original ones so we may choose to work with whichever is most convenient at any given moment.

In order to study the field $|\varphi_n(e^{i \theta})|^{-2}$ we will still need to keep track of the $B_n$ field at other points. Given that $B_n(1)$ is already absorbed into the modified Verblunskies, it is enough to keep track of the ratio $B_n(e^{i \theta})/B_n(1)$. Note that this is an evolution on $\bd$ so we may define a sequence of continuous, increasing functions $\delta_n : (-\pi, \pi) \to \R$ such that 
\[
 e^{ i \delta_n(\theta) } = \frac{B_n(e^{i \theta})}{B_n(1)}. 
\]
This quantity is known as the \textbf{relative Pr\"{u}fer phase} at time $n$. Strictly speaking this only defines the process modulo $2 \pi \Z$, so the more precise definition we shall use is that $\delta_n$ is defined by the recurrence relation
\begin{equation} \label{eqn:relative_Prufer_phase}
\delta_{n+1}(\theta) = \delta_n(\theta) + \theta + 2 \Im \log \left( \frac{1 - \gamma_n}{1 - \gamma_n e^{i \delta_n(\theta)}} \right), \quad \delta_0(\theta) = \theta.
\end{equation}
This recursion is a straightforward consequence of \eqref{eqn:V_recursion}. The branch of the logarithm is chosen so as to give $0$ when $\gamma_n = 0$, see \cite[Proposition 2.2]{KillipStoiciu} for more details. Note that, by its definition, $\delta_n$ is a measurable function of $\gamma_0, \dots, \gamma_{n-1}$, and that $\delta_n(0) = 0$ for all $n$. These definitions lead to the commonly used relationship
\begin{equation} \label{eq:akBk}
\alpha_k B_k(e^{i \theta}) = \gamma_k e^{i \delta_k(\theta)}.
\end{equation}

\subsection{Caratheodory Functions and the Poisson Kernel}

The Poisson kernel for the disk is useful for directly studying $\mu$ and for the recurrence relation \eqref{eqn:phi_psi_mod_recursion}. We denote the disk Poisson kernel by $P_{\D}$, and recall that it is given by
\begin{equation}
P_{\D}(z, e^{i \theta}) = \frac{1-|z|^2}{|e^{i \theta} - z|^2}
\end{equation}
for $z \in \D$. Clearly $P_{\D}(z, e^{i \theta}) = P_{\D}(z e^{-i \theta}, 1)$ and recall that $z \mapsto P_{\D}(z, 1)$ is harmonic since
\begin{equation} \label{eqn:poisson_kernel_real_part}
P_{\D}(z, 1) = \Re \left ( \frac{1+z}{1-z} \right ).
\end{equation}
For $z_0 \in \D$, let $\phi_{z_0} : \D \to \D$ be the conformal map
\begin{equation} \label{eqn:mobius_map}
\phi_{z_0}(z) = \frac{z - z_0}{1 - \overline{z_0} z}.
\end{equation}
Recall that this is the unique conformal map of the disk to itself which sends $z_0$ to $0$ and has positive derivative at the origin. It is straightforward to check the relation
\begin{equation} \label{eqn:poisson_modulus}
| \phi_{z_0}'(e^{i \theta}) | = P_{\D} \left ( z_0, e^{i \theta} \right ).
\end{equation}
Combining the recursion \eqref{eqn:phi_psi_mod_recursion} and \eqref{eq:akBk} with these properties of the Poisson kernels gives
\begin{equation} \label{eqn:phi_psi_prod_formula}
|\varphi_n(e^{i \theta})|^{-2} = \prod_{k=0}^{n-1} P_{\D}(\alpha_k, \overline{B_k(e^{i \theta})}) = \prod_{k=0}^{n-1} P_{\D}(\gamma_k e^{i \delta_k(\theta)}, 1).
\end{equation}
We will often take $\theta = 0$ so that the terms on the right hand side are simply $P_{\D}(\gamma_k, 1)$.

The formulas above are useful when the Bernstein-Szeg\H{o} approximation is used to approximate $\mu$, but the Poisson kernel can also be used to study $\mu$ by considering the Caratheodory function:
\[
F(z) = F(z; \dmu) := \int_{\bd} \frac{e^{i \theta}+z}{e^{i \theta}-z} \; \dmu(\theta).
\]
Note that the real part of the integrand is the Poisson kernel. The behavior of the modulus $|F(r e^{i \theta})|$ as $r \to 1$ gives information on the behavior of the measure, and its magnitude can be bounded in terms of orthogonal polynomials using the \textit{Jitomirskaya-Last inequalities}.

\begin{theorem}[Jitomirskaya-Last inequalities \cite{JitoLast:Acta, Simon2}]\label{thm:Jito_Last}
There exists a universal constant $A > 1$ (independent of the choice of probability measure $\mu$ on $\bd$) such that
\[
A^{-1} \frac{\|\varphi_{\cdot}(e^{i \theta})\|_{x(r)}}{\|\psi_{\cdot}(e^{i \theta})\|_{x(r)}} \leq |F(r e^{i \theta})| \leq A \frac{\|\varphi_{\cdot}(e^{i \theta})\|_{x(r)}}{\|\psi_{\cdot}(e^{i \theta})\|_{x(r)}}
\]
for any $r \in [0,1)$ and $e^{i \theta} \in \bd$. The norm is the following: for a sequence $(a_n)_{n \geq 0}$ in $\C$ and $x > 0$ we define $\|a_{\cdot}\|_{x}$ by
\[
\|a_{\cdot}\|_{x}^2 = \sum_{n=0}^{\lfloor x \rfloor} |a_n|^2 + (x - \lfloor x \rfloor)|a_{\lceil x \rceil}|^2.
\]
Then for $r \in [0,1)$ the function $x(r)$ is the unique solution to
\[
(1-r) \|\varphi_{\cdot}(e^{i \theta})\|_{x(r)} \|\psi_{\cdot}(e^{i \theta})\|_{x(r)} = \sqrt{2}.
\]
\end{theorem}

It is known that at least one of the sequences $(\vphi_n(e^{i \th}))$ or $(\psi_n(e^{i \th}))$ is not in $\ell^2(\N)$, hence the function $x(r)$ is well-defined. The original proof of these inequalities in \cite{JitoLast:Acta} was for orthogonal polynomials on the line. The translation to the OPUC case is in \cite{Simon2}. 

%
%
%
%

\subsection{Dimension Theory \label{sec:dimension_theory}}

For singular continuous measures there are many different methods for describing their fractal geometry. We use the very precise technique that is described in \cite{Last:SingContSpectra, JitoLast:Acta}, which in turn comes from the earlier work of Rogers and Taylor \cite{RogersTaylor:ActaI, RogersTaylor:ActaII}. The Rogers and Taylor theory provides a decomposition of a singular continuous measure with respect to a Hausdorff measure, akin to the standard Lebesgue decomposition of a general measure into components that are atomic, absolutely continuous, and singular continuous.

To understand the theory, briefly recall the definition of Hausdorff measure: for $d \in [0,1]$ and $S \subset \bd$ define $h^{d}(S)$ by
\[
h^{d}(S) = \lim_{\delta \downarrow 0} \inf_{\delta-\textrm{covers}} \sum_{i=1}^{\infty} |E_i|^{d},
\]
where the infimum is over all covers of $S$ by arcs of length less than $\dl$.
When restricted to Borel sets, $h^{d}$ is a Borel measure. For any such $S$ there is a number $\hdim S$, called the Hausdorff dimension of $S$, such that $d > \hdim S \implies h^{d}(S) = 0$ and $d < \hdim S \implies h^{d}(S) = \infty$. Conversely $h^{d}(S) = 0 \implies \hdim S \leq d$ and $h^{d}(S) = \infty \implies \hdim S \geq d$. Based on these notions, Rogers and Taylor make the following definitions.

\begin{definition}
For $d \in [0,1]$, a probability measure $\mu$ on $\bd$
\begin{enumerate}[(i)]
\item is $d$-continuous if $\mu(S) = 0$ for all $S \subset \bd$ with $h^{d}(S) = 0$,
\item is $d$-singular if there exists $S \subset \bd$ such that $\mu(S) = 1$ but $h^{d}(S) = 0$,
\item has exact dimension $d$ if for every $\epsilon > 0$ it is both $(d - \epsilon)$-continuous and $(d + \epsilon)$-singular.
\end{enumerate}
\end{definition}

The exact dimension of a measure is a very useful description of the size of the fractal sets that support it. To compute exact dimensions for a measure, Rogers and Taylor introduce the notion of local dimension.

\begin{definition}
Given a probability measure $\mu$ on $\bd$ and a parameter $s \in [0,1]$, define, for each $e^{i \theta} \in \bd$,
\[
D_{\mu}^s(\theta) := \limsup_{\epsilon \downarrow 0} \frac{\mu((\theta - \epsilon, \theta + \epsilon))}{(2 \epsilon)^s},
\]
where $(\theta - \epsilon, \theta + \epsilon)$ refers to the arc of length $2 \epsilon$ on $\bd$ centered around $e^{i \theta}$. Then for each $\theta$, there clearly exists an $s_0 \in [0,1]$ such that
\[
D_{\mu}^s(\theta) =
\begin{cases}
0 & \text{if $s < s_0$,} \\
\infty & \text{if $s > s_0$.}
\end{cases}
\]
We will refer to this $s_0 = s_0(\mu, \theta)$ as the \textit{local dimension} of the measure $\mu$ at the point $\theta$. 
%
\end{definition}

With this definition in hand, Rogers and Taylor produce the following important result.

\begin{theorem}[Rogers-Taylor \cite{RogersTaylor:ActaI, RogersTaylor:ActaII}]\label{thm:rogers_taylor}
If the local dimension for $\mu$ is $s_0$ at $\mu$-a.e. points $\theta$, then $\mu$ has exact dimension $s_0$.
\end{theorem}

The final result, which is again due to Jitomirskaya and Last and is a consequence of their inequalities (Theorem \ref{thm:Jito_Last}), gives a way of computing the local dimension via orthogonal polynomials. In subsequent sections this will be one of our main tools.

\begin{proposition}[Jitomirskaya-Last \cite{JitoLast:Acta, Simon2}]\label{prop:local_dim_test}
For $s \in (0,1)$ let $t = s/(2-s)$. Then
\begin{align*}
s > s_0(\mu, \theta) \iff D_{\mu}^s(\theta) = \infty \iff \liminf_{x \to \infty} \frac{\|\varphi_{\cdot}(e^{i \theta})\|_x}{\|\psi_{\cdot}(e^{i \theta})\|_{x}^{t}} = 0.
\end{align*}
\end{proposition}

An important corollary of this fact, which we will use repeatedly, is the following. The same argument appears in \cite{Simon2}.

\begin{corollary}\label{corollary:local_dim_polynomials}
If, for a probability measure $\mu$ on $\bd$ and a point $e^{i \theta} \in \bd$, there are constants $c, d < 1$ such that
\[
\lim_{n \to \infty} \frac{\log |\varphi_n(e^{i \theta})|^{-2}}{\log n} = c, \quad \lim_{n \to \infty} \frac{\log |\psi_n(e^{i \theta})|^{-2}}{\log n} = d,
\]
then $s_0(\mu, \theta) = 2(1-c)/(2-c-d)$.
\end{corollary}

Note that the results are expressed in terms of the negative second power of the orthogonal polynomials, which turns out to be the more useful quantity for random Verblunksy coefficients.

\begin{proof}
By the assumptions, for any $\epsilon > 0$ there exists an $N$ large such that $n \geq N$ implies
\[
n^{c - \epsilon} \leq |\varphi_n(e^{i \theta})|^{-2} \leq n^{c + \epsilon}, \quad n^{d - \epsilon} \leq |\psi_n(e^{i \theta})|^{-2} \leq n^{d + \epsilon}.
\]
Thus by summing and using the definition of the $\|\cdot\|_x$ norms and the fact that $c, d < 1$ (so that the norms grow with $x$) there are constants $C_0, C_1$ such that
\[
C_0 x^{1 - c - \epsilon} \leq \|\varphi_{\cdot}(e^{i \theta})\|_x^2 \leq C_1 x^{1 - c + \epsilon}, \quad C_0 x^{1 - d - \epsilon} \leq \|\psi_{\cdot}(e^{i \theta})\|_x^2 \leq C_1 x^{1-d+\epsilon}.
\]
From these inequalities it is straightforward to check that if $t = s/(2-s)$ then
\[
\liminf_{x \to \infty} \frac{\|\varphi_{\cdot}(e^{i \theta})\|_x}{\|\psi_{\cdot}(e^{i \theta})\|_x^{t}} =
\begin{cases}
\infty, & s < s_0 \\
0, & s > s_0
\end{cases},
\]
with $s_0 = 2(1-c)/(2-c-d)$. Proposition \ref{prop:local_dim_test} completes the proof.
\end{proof}


\section{Rotationally Invariant Independent Verblunsky Coefficients \label{sec:RIVB}}

\subsection{Construction}

From the last subsection we see that dimensional properties of a probability measure can be derived by analyzing the random variable $D_{\mu}^s(\theta)$, where $\theta$ is a random point on $\bd$ distributed according to $\mu$. We are interested in the case where $\mu$ itself is random and constructed by making a random choice for the Verblunsky coefficients. Letting $Q$ be a probability measure on $\D^{\infty}$ and $\bfalpha = \{ \alpha_n, n \geq 0 \}$ be a generic element of $\D^{\infty}$, we are therefore interested in the joint probability measure
\begin{equation} \label{eqn:joint_prob_measure}
\rmd \P(\bfalpha, \theta) := \rmd Q(\bfalpha) \: \dmu_{\bfalpha}(\theta)
\end{equation}
on $\D^{\infty} \times \bd$, where $\mu_{\bfalpha}$ is the spectral measure corresponding to the sequence $\bfalpha$. If the pair $(\bfalpha, \theta)$ is distributed according to this measure, then by Section \ref{sec:dimension_theory} dimensional properties for $\mu_{\bfalpha}$ can be derived by studying the random variable $D_{\mu_{\bfa}}^s(\theta)$.

It is worthwhile to be slightly pedantic about this construction. Throughout our probability space will be $\Omega = \D^{\infty} \times \bd$ equipped with its Borel $\sigma$-algebra $\F$ (under the product topology). By the Bernstein-Szeg\H{o} approximation \eqref{eqn:berstein_szego_approx}, the mapping $\bfa \mapsto \mu_{\bfa}$ is measurable from $\Dinf$ to the space of probability measures on $\bD$ with the weak topology. In other words, $\mu$ is a random measure in the usual sense, see \cite{KallenbergRM}. Then, given a probability measure $Q$ on $\D^{\infty}$, the probability measure $\P$ is defined on rectangles by
\[
\P(A \times B) = \int_A \mu_{\bfalpha}(B) \: \rmd Q(\bfalpha) = \E_Q \left ( \mu_{\bfalpha}(B) \onef_{\{ \bfalpha \in A \}} \right )
\]
for Borel sets $A \subset \D^{\infty}, B \subset \bd$. Finally $\P$ is extended to all of $\F$ in the usual way. Since the spaces $\D^{\infty}$ and $\bd$ are Polish there are regular conditional distributions for both the $\bfalpha$ and $\theta$ variables, see \cite[Theorem 5.3]{KallenbergFMP} for a proof of this fact. By the construction of $\P$ the conditional distribution of $\theta$ given $\bfalpha$ is exactly $\mu_{\bfalpha}$, while the marginal distribution of $\bfalpha$ is simply $Q$. In the next section we will analyze the conditional distribution $Q_{\theta}$ of the Verblunsky coefficients $\bfalpha$ given $\theta$. We will also make heavy use of the measure
\begin{equation} \label{eqn:ind_joint_measure}
\rmd \Pl(\bfalpha, \theta) = \rmd Q(\bfalpha) \, \frac{\dth}{2 \pi},
\end{equation}
under which the Verblunsky coefficients and $\theta$ are independent.

Note that under the joint measure $\P$, we have the following straightforward extension of the Rogers-Taylor result, Theorem \ref{thm:rogers_taylor}.

\begin{corollary}\label{cor:P_dimension}
If the local dimension for $\mu_{\bfalpha}$ is $s_0$ with $\P$-probability one, then $\mu_{\bfalpha}$ has exact dimension $s_0$ for $Q$-a.e. realizations of the Verblunsky coefficients $\bfalpha$.
\end{corollary}

\begin{proof}
By the definition of exact dimension, the hypothesis is that for that any $\epsilon > 0$
\[
\P \left ( \bigcap_{\epsilon > 0} \{ (\bfalpha, \theta) : D_{\mu_{\bfalpha}}^{s_0 - \epsilon}(\theta) = 0, D_{\mu_{\bfalpha}}^{s_0 + \epsilon}(\theta) = \infty \} \right ) = 1.
\]
By definition of $\P$ the latter probability is simply
\[
\int_{\D^{\infty}} \mu_{\alpha} \left ( \bigcap_{\epsilon > 0} \{ \theta : D_{\mu_{\bfalpha}}^{s_0 - \epsilon}(\theta) = 0, D_{\mu_{\bfalpha}}^{s_0 + \epsilon}(\theta) = \infty \} \right ) \: \rmd Q(\bfalpha) = 1.
\]
But since the integrand is always less than or equal to one (since $\mu_{\bfalpha}$ is a probability measure) the integral can equal one only if the integrand is one $Q$-almost surely.
\end{proof}

\subsection{Rotationally Invariant Independent Modified Verbunsky coefficients}

\textbf{From this point on, we enforce Assumption \ref{as:indrot}}, namely that the Verblunsky coefficients $(\a_k)$ are independent and with a rotationally symmetric distribution. The first easy but extremely useful result concerns the modified Verblunsky coefficients \eqref{eq:modified_Verblunsky}. Recall that they are defined in Section \ref{sec:modifiedVerb} and allow to simplify many formulas. Under Assumption \ref{as:indrot}, using the modified instead of the original coefficients is very convenient because of the following.

\begin{lemma}(\cite[Lemma 2.1]{KillipStoiciu}) \label{lemma:KS_modified_Verblunsky_law}
The sequence of modified Verblunsky coefficients $(\gamma_k, k \geq 0)$ has the same law as the original sequence of Verblunsky coefficients $(\alpha_k, k \geq 0)$.
\end{lemma}

This result may seem slightly counterintuitive at first since $B_k(1)$ is itself a function of the $\alpha_k$, but the important properties is that $B_k(1)$ is a point on $\bd$ that is measureable with respect to $\alpha_0, \ldots, \alpha_{k-1}$, and hence, conditionally on $\alpha_0, \ldots, \alpha_{k-1}$, the quantity $\gamma_k$ is a fixed rotation of $\alpha_k$ and hence has the same law, independent of the actual value of $\alpha_0, \ldots, \alpha_{k-1}$. As a consequence we may choose to work with the modified Verblunsky coefficients instead of the original ones without changing the most important results. We will do so in most of the rest of the paper, although we will freely move between them at many different points. 

\subsection{Martingale Properties of Orthogonal Polynomials}

In this section we describe some particular martingales that arise from orthogonal polynomials and their relation to the Bernstein-Szeg\H{o} approximation.

\begin{lemma}\label{lemma:joint_rotation}
For each fixed $\lambda \in \bd$ the measure $\P$ is invariant under the transformation
\[
\left( \{ \alpha_n \}_{n=0}^{\infty}, e^{i \theta} \right) \mapsto \left( \{ \lambda^{n+1} \alpha_n \}_{n=0}^{\infty}, \overline{\lambda} e^{i \theta} \right).
\]
\end{lemma}

\begin{proof}
By assumption, the sequences $\{ \alpha_n, n \geq 0 \}$ and $\{ \lambda^{n+1} \alpha_n, n \geq 0 \}$ have the same law, hence so too do the associated random measures. By Lemma \ref{lemma:rotated_Verblunskies} these measures are $\mu_{\bfalpha}(\cdot)$ and $\mu_{\bfalpha}(\lambda \cdot)$, hence the result follows.
\end{proof}

\begin{proposition}\label{prop:phi_martingale_prop}
Under the measure $\Pl$, the process $n \mapsto |\varphi_n(e^{i \theta})|^{-2}$ is a martingale with respect to the filtration $\F_{n-1}^{\theta} := \sigma(\theta, \alpha_0, \alpha_1, \ldots, \alpha_{n-1})$.
\end{proposition}

\begin{proof}
First observe that by the recursions \eqref{eqn:phi_psi_prod_formula} and \eqref{eqn:relative_Prufer_phase} the variables $\varphi_n(e^{i \theta})$ and $\delta_n(\theta)$ are measurable with respect to $\F_{n-1}^{\theta}$ (the shift in the indexing is because $\varphi_0$ and $\delta_0$ are deterministic by convention, while $\varphi_1$ and $\delta_1$ are determined by the Verblunsky coefficient $\alpha_0$, and so on). Using this measurability and recursion \eqref{eqn:phi_psi_prod_formula} we have the relation
\[
\E_{\Pl} \left [ \left. | \varphi_{n+1}(e^{i \theta})|^{-2} \right | \F_{n-1}^{\theta} \right ] = |\varphi_n(e^{i \theta})|^{-2} \E_{\Pl} \left[ \left. P_{\D}(\gamma_n e^{i \delta_n(\theta)}, 1) \right| \F_{n-1}^{\theta} \right].
\]
In the remaining expectation the term $\delta_n(\theta)$ is $\F_{n-1}$-measurable, whereas $\gamma_n$ is independent of $\F_{n-1}$. Recall as well that $\gamma_n$ has a rotationally invariant law, by Lemma \ref{lemma:KS_modified_Verblunsky_law}. Therefore
\[
\E_{\Pl} \left[ \left. P_{\D}(\gamma_n e^{i \delta_n(\theta)}, 1 ) \right| \F_{n-1}^{\theta} \right] = \E_{\Pl} \left[  P_{\D}(\gamma_n, 1 ) \right] = \E_Q \left[ P_{\D} (0,1) \right] = 1.
\]
The second-to-last equality uses that $\gamma_n$ is rotation invariant and that $z \mapsto P_{\D}(z, 1)$ is harmonic, so that its integral on a circle is just the value at the center, i.e. $P_{\D}(0,1) = 1$.
\end{proof}

\begin{remark}\label{rmk:filtration}
To be pedantic again, the filtration $\F_n^{\theta}$ is the Borel $\sigma$-algebra generated by the $\theta$ variable and the first $n+1$ elements of a sequence in $\D^{\infty}$, i.e.
\[
\F_n^{\theta} = \sigma \left( \bigtimes_{i=0}^{n} \mathcal{B}(\D) \times \bigtimes_{i=n+1}^{\infty} \{ \emptyset, \D \} \times \mathcal{B}(\bd) \right),
\]
where $\mathcal{B}(\D)$ is the Borel $\sigma$-algebra for $\D$, $\mathcal{B}(\bd)$ is the Borel $\sigma$-algebra for $\bd$, and $\{\emptyset, \D \}$ is the trivial $\sigma$-algebra for $\D$. Note that these $\F_n^{\theta}$ generate the entire $\sigma$-algebra for $\Omega$, i.e.
\[
\F = \sigma \left( \cup_{n=1}^{\infty} \F_n^{\theta} \right).
\]
For consistency of notation we write $\F_{-1}^{\theta}$ to be the $\sigma$-algebra for $\theta$, i.e $\F_{-1}^{\theta} = \sigma(\times_{i \geq 0} \{\emptyset, \D \} \times \mathcal{B}(\bd) )$. Forgetting the $\theta$ variable entirely we can also consider the smaller $\sigma$-algebras
\[
\F_n = \sigma \left( \bigtimes_{i=0}^n \mathcal{B}(\D) \times \bigtimes_{i=n+1}^{\infty} \{ \emptyset, \D \} \times \{ \emptyset, \bd \} \right).
\]
Slightly abusing notation, we will also think of $(\F_n)$ as the standard filtration for $\Dinf$.

\end{remark}

\begin{remark}\label{rmk:weak_martingale_result}
There is a slightly weaker version of Proposition \ref{prop:phi_martingale_prop} that at first may seem more natural: for any fixed $\theta$ a virtually identical argument shows that $|\varphi_n(e^{i \theta})|^{-2}$ is a martingale with respect to $\F_n$, under the measure $Q$. This gives the relation
\begin{equation} \label{eqn:weak_martingale_relation}
\E_{Q} \left[ \left. |\varphi_{n+1}(e^{i \theta})|^{-2} \right| \F_{n-1} \right] = |\varphi_n(e^{i \theta})|^{-2} \quad Q-\textrm{a.s.}
\end{equation}
However note that there is a null set involved in this identity, and by fixing $\theta$ the null set may depend on $\theta$. The union of these null sets over the uncountably many $\theta$ may become an event of positive probability. Proposition \ref{prop:phi_martingale_prop} avoids this by working with respect to the joint measure and gives that there is only one single null set in $\D^{\infty} \times \bd$ on which the martingale relation does not hold. Also note that in the proof of Proposition \ref{prop:phi_martingale_prop} the Lebesgue measure $\rmd \theta/2\pi$ could have been replaced by any fixed probability measure on $\bd$. This determines what measure the null set is with respect to, but for our purposes, the Lebesgue measure will be sufficient.
\end{remark}

The latter remark is useful for the next result, which is noteworthy because it provides a probabilistic proof of the convergence part of the Bernstein-Szeg\H{o} approximation.

\begin{corollary}\label{cor:martingale_bernstein_szego}
For independent and rotationally invariant Verblunsky coefficients, the probability measures $|\varphi_n(e^{i \theta})|^{-2} \, \rmd \theta/2 \pi$ almost surely converge (weakly) as $n \to \infty$.
\end{corollary}

\begin{proof}
Let $f : \bd \to \R$ be continuous and hence bounded. Using Proposition \ref{prop:phi_martingale_prop}, it is straightforward to show that the process
\begin{align*}
n \mapsto \int_{\bd} f(e^{i \theta}) |\varphi_n(e^{i \theta})|^{-2} \: \frac{\dth}{2 \pi}
\end{align*}
is a martingale with respect to $Q$ and $(\F_{n-1}, n \geq 0)$, and since it is bounded (because $f$ is and $|\varphi_n|^{-2} \dth/2\pi$ is a probability measure) the martingale convergence theorem implies that it converges almost surely. Since the almost sure convergence holds simultaneously for a countable and dense collection of $f$, this implies a.s. weak convergence (see \cite{KallenbergRM}).

\end{proof}

Note that this only provides a probabilistic proof of the convergence part of the Bernstein-Szeg\H{o} approximation. We do not have a separate proof that $|\varphi_n(e^{i \theta})|^{-2} \: \dth/2\pi$ are probability measures (which is implicitly used in the above proof), nor are we able to identify the limiting measure as being the spectral measure. However, the Bernstein-Szeg\H{o} approximation implies that the conditional expectation of the spectral measure $\mu_{\bfalpha}$ given $\F_n$ is the measure $|\varphi_{n+1}(e^{i \theta})|^{-2} \: \dth/2 \pi$, in the following sense.

\begin{lemma}\label{lem:PAB}
Let $F : \D^{\infty} \times \bd \to \R$ be continuous and $\F_n^{\theta}$-measurable. Then for $n \geq -1$
\begin{align*}
\E_{\P} \left[ \left. F(\bfalpha, \theta) \right| \F_n \right] = \E_Q \left[ \left. \int_{\bd} F(\bfalpha, \theta) \: \rmd \mu_{\bfalpha}(\theta) \right| \F_n \right] = \int_{\bd} F(\bfalpha, \theta) |\varphi_{n+1}(e^{i \theta})|^{-2} \: \frac{\dth}{2 \pi}.
\end{align*}
In particular, for $A \in \F_n$ and $B \in \mathcal{B}(\bd)$ we have
\begin{align*}
\P(A \times B) = \int_{A} \int_{B} |\varphi_{n+1}(e^{i \theta})|^{-2} \, \frac{\dth}{2 \pi} \: \rmd Q(\bfalpha).
\end{align*}
\end{lemma}

The proof is very similar to that of Corollary \ref{cor:martingale_bernstein_szego}. Finally we note that there are probabilistic proofs concerning the absolute continuity and singularity of the spectral measure $\mu_{\bfalpha}$ with respect to Lebesgue measure.

\begin{proposition}\label{prop:spectral_lebesgue_decomp}
With respect to the measure $\rmd \Pl$,
\begin{enumerate}[(i)]
\item the event $\left \{ |\varphi_n(e^{i \theta})|^{-2} \to 0 \right \}$ has either probability 1 or probability 0;
\item if $|\varphi_n(e^{i \theta})|^{-2} \to 0$ with probability 1, then with $Q$-probability one, $\mua$ is singular with respect to Lebesgue measure;
\item if $|\varphi_n(e^{i \theta})|^{-2}$ converges in $L^1$, then with $Q$-probability 1, $\mua$ is absolutely continuous with respect to the Lebesgue measure.
\end{enumerate}
\end{proposition}

\begin{proof}
For ease of notation write $X(\th) = \limsup_{n \to \infty} |\varphi_n(e^{i \theta})|^{-2}$. By the martingale property of Proposition \ref{prop:phi_martingale_prop} we know that $|\varphi_n(e^{i \theta})|^{-2}$ converges $\Pl$-almost surely as $n \to \infty$, and in particular $\Pl(X < \infty) = 1$. To prove part (i), note that the event $\{X = 0\}$ is a tail event for the filtration $\F_n^{\theta}$ (by, for example, equation \eqref{eqn:phi_psi_prod_formula} and the positivity of the Poisson kernel), and since $\mathbf{P}$ is a product measure, we have by the Kolmogorov 0-1 Law that $\mathbf{P}(X = 0) \in \{0,1 \}$.

For parts (ii) and (iii) first define measures $\P_{ac}$ and $\P_s$ by
\begin{align*}
\rmd\P_{ac} = X \: \rmd \Pl, \quad \P_s(A) = \P(A \cap \{ X = \infty \}),
\end{align*}
for Borel $A \subset \Om$. Since $\Pl(X < \infty) = 1$, we have $\P_{ac} \ll \Pl$ and $\P_s \perp \mathbf{P}$ (hence the reason for the names). By classical arguments \cite[Theorem 5.3.3]{Durrett:book}, we have that $\P = \P_{ac} + \P_s$, which is the Lebesgue decomposition of $\P$ with respect to $\mathbf{P}$. In case (ii) we have that $\Pl(X=0) = 1$ and thus $\P$ is singular with respect to $\Pl$. Then, by a straightforward Fubini argument, we see that $\Leb(\{ e^{i \theta} \in \bd : (\bfa, e^{i \theta}) \in \Om_0 \}) = 0$, $Q$-a.e., as desired. Part (iii) is handled similarly.

\end{proof}

\subsection{The Conditional Measure for Verblunsky Coefficients}

Now we turn to the task of describing the conditional distribution $Q_{\theta}$ for the joint measure $\P$. The intuition comes from the Bernstein-Szeg\H{o} approximation, which says that
\[
\rmd \P(\bfalpha, \theta) = \rmd Q(\bfalpha) \: \rmd \mua(\theta) = \lim_{n \to \infty} \rmd Q(\bfalpha) |\varphi_n(e^{i \theta})|^{-2} \dth/2\pi.
\]
The next lemma shows that the marginal distribution of $\theta$ is Lebesgue, and therefore the formula above suggests that
\[
\rmd Q_{\theta}(\bfalpha) = \lim_{n \to \infty} |\varphi_n(e^{i \theta})|^{-2} \: \rmd Q(\bfalpha).
\]
A proof of this relation follows after. It consists of showing that the limit above exists and that it agrees with the conditional measure.

\begin{lemma}\label{lemma:theta_marginal}
Under the measure $\P$ the marginal distribution of $\theta$ is Lebesgue.
\end{lemma}

\begin{proof}[Proof 1]
Fix $\lambda \in \bd$. By Lemma \ref{lemma:joint_rotation} the random variables $e^{i \theta}$ and $\lambda e^{i \theta}$ have the same law under $\P$, i.e. the law is invariant under all rotations. Therefore it must be Lebesgue.
\end{proof}

\begin{proof}[Proof 2]
By definition of $\P$ the marginal measure for $\theta$ is the probability measure $\nu$ on $\bd$ with
\[
\nu(B) = \P(\D^{\infty} \times B) = \E_{Q} \left[ \mu_{\bfalpha}(B) \right]
\]
for all Borel subsets $B$ of $\bd$. However, since $\D^{\infty} \in \F_{-1}$ the second part of Lemma \ref{lem:PAB} implies
\[
\P(\D^{\infty} \times B) = \int_{\D^{\infty}} \int_{B} |\varphi_0(e^{i \theta})|^{-2} \: \frac{\dth}{2 \pi} \: \rmd Q(\bfalpha) = \int_{B} \frac{\dth}{2 \pi}.
\]
The last equality follows from $\varphi_0 = 1$.
\end{proof}

Note the first proof is based on independence and rotation invariance of the Verblunsky coefficients while the second uses the martingale property. If the martingale property can be established without assuming independence and rotation invariance then the second proof still goes through.

\begin{proposition}\label{prop:Q_theta_exists}
For each $\theta$, the limiting measure
\[
\rmd Q_{\theta}(\bfalpha) := \lim_{n \to \infty} |\varphi_n(e^{i \theta})|^{-2} \: \rmd Q(\bfalpha)
\]
exists, and, under $\P$, is the conditional measure of $\bfa$ given $\theta$.
\end{proposition}

\begin{proof}
The proof of existence is standard. Let $A \in \F_n$. Then by the martingale property \eqref{eqn:weak_martingale_relation},
\[
\E_Q \left[ \1{\bfalpha \in A} |\varphi_m(e^{i \theta})|^{-2} \right] = \E_Q \left[ \1{\bfalpha \in A} |\varphi_{n+1}(e^{i \theta})|^{-2} \right]
\]
for all $m \geq n+1$. Thus the sequence of measures $|\varphi_n(e^{i \theta})|^{-2} \: \rmd Q(\alpha)$, each one defined on $\F_n$, form a consistent family, and therefore, by the Kolmogorov Extension Theorem \cite[Chapter 1, Theorem 3.2]{RevuzYor}, they extend uniquely to a measure $Q_{\theta}$ on $\sigma(\cup_{n \geq 0} \F_n) = \mathcal{B}(\D^{\infty})$. Furthermore, by the equality above, we have
\begin{align*}
Q_{\theta}(A) = \lim_{m \to \infty} \E_Q \left[ \1{\bfalpha \in A} |\varphi_m(e^{i \theta})|^{-2} \right] = \E_{Q} \left[ \1{\bfalpha \in A} | \varphi_{n+1}(e^{i \theta})|^{-2} \right].
\end{align*}
This is sufficient to show that $Q_{\theta}$ is the weak limit of these measures \cite[Theorem 2.2]{Billingsley:convergence_of_measures}.

To show that $Q_{\theta}$ is the conditional distribution of $\bfalpha$, first observe that the fact that it can be written as a limit implies that $\theta \mapsto Q_{\theta}(A)$ is measurable for each $A \in \mathcal{B}(\D^{\infty})$. Hence it makes sense to consider the measure $\P'$ on $\Omega$ defined by
\begin{align*}
\P'(A \times B) = \int_{B} Q_{\theta}(A) \: \frac{\dth}{2 \pi}
\end{align*}
for $A \in \mathcal{B}(\D^{\infty})$, and $B \in \mathcal{B}(\bd)$. Now if $A \in \F_n$ then $Q_{\theta}(A) = \E_Q [|\varphi_{n+1}(e^{i \theta})|^{-2} \1{\bfalpha \in A}]$, by the above definition of $Q_{\theta}$, hence
\begin{align*}
\P'(A \times B) = \int_{B} \int_A |\varphi_{n+1}(e^{i \theta})|^{-2} \: \rmd Q(\bfalpha) \: \frac{\dth}{2 \pi} = \int_A \int_B |\varphi_{n+1}(e^{i \theta})|^{-2} \: \frac{\dth}{2 \pi} \: \rmd Q(\bfalpha) = \P(A \times B).
\end{align*}
The last equality is the second part of Lemma \ref{lem:PAB}. Thus $\P$ and $\P'$ agree on all sets of this form, and so by a classical $\pi$-$\lambda$ system type argument \cite{Durrett:book}, they are the same.
\end{proof}

From the definition of $Q_{\theta}$, we can quickly derive some simple properties.

\begin{proposition}\label{prop:Q_theta_properties}
The following statements hold.
\begin{enumerate}[(i)]
\item If $|\varphi_n(e^{i \theta})|^{-2} \to 0$ with positive $Q$-probability, then the measures $Q$ and $Q_{\theta}$ are singular.
\item If $|\varphi_n(e^{i \theta})|^{-2}$ converges in $L^1$, then the measures $Q$ and $Q_{\theta}$ are absolutely continuous.
\item For any $\th \in \bd$, if $\{ \alpha_n \}_{n=0}^{\infty}$ has law $Q_{\theta}$, then $\{e^{i(n+1)\theta} \alpha_n \}_{n=0}^{\infty}$ has law $Q_0$.
\item If $(\{ \alpha_n \}_{n=0}^{\infty}, \theta)$ has law $\P$, then $\{e^{i(n+1)\theta} \alpha_n \}_{n=0}^{\infty}$ has law $Q_0$.
\end{enumerate}
\end{proposition}

\begin{proof}
For parts (i) and (ii), let $X(\theta) = \limsup_{n \to \infty} |\varphi_n(e^{i \theta})|^{-2}$. For each fixed $\th$, the fact that $(|\varphi_n(e^{i \theta})|^{-2})$ is a martingale shows that $Q$-a.s., the limsup above is merely a limit. In particular $Q(X(\theta) < \infty) = 1$. We can then define measures $Q_{\theta, \textrm{ac}}$ and $Q_{\theta, \textrm{s}}$ by
\begin{align*}
\rmd Q_{\theta, ac} = X(\theta) \: \rmd Q, \quad Q_{\theta, s}(A) = Q_{\th}(A \cap \{ X(\theta) = \infty \})
\end{align*}
and conclude as in the proof of Proposition \ref{prop:spectral_lebesgue_decomp}.

For part (iii), we have that, for any $F$ that is $\F_{n-1}$-measurable,
\begin{align*}
\E_{Q_{\theta}}[F(\bfalpha)] &= \E_Q[|\varphi_n(e^{i \theta}; \bfalpha)|^{-2} F(\bfalpha)] \\
&= \E_Q[|\varphi_n(e^{i \theta}; \{ e^{-i (k+1) \theta} \alpha_k \})|^{-2} F(\{ e^{-i(k+1) \theta} \alpha_k \})] \\
&= \E_Q[|\varphi_n(1; \bfalpha)|^{-2} F(\{ e^{-i(k+1) \theta} \alpha_k \})] = \E_{Q_0}[F(\{ e^{-i(k+1) \theta} \alpha_k \})].
\end{align*}
The first equality is the definition of $Q_{\th}$ (see the previous proof), the second holds by rotation invariance of $Q$, and the third is by Lemma \ref{lemma:rotated_Verblunskies}. Finally, part (iv) follows from (iii) and the decomposition $\rmd \P(\bfalpha, \theta) = \rmd Q_{\theta}(\bfalpha) \: \dth/2\pi$, since for Borel $A \subset \D^{\infty}$,
\[
\P \left( \{e^{i(n+1) \theta} \alpha_n \}_{n=0}^{\infty} \in A \right) = \int_{\bd} Q_{\theta} \left( \{e^{i(n+1) \theta} \alpha_n \}_{n=0}^{\infty} \in A \right) \, d \theta/2\pi = Q_0 \left( \{ \alpha_n \}_{n=0}^{\infty} \in A \right).
\]
\end{proof}

Parts (iii) and (iv) show that it is enough to study the measures $Q_0$. In particular, part (iii) combined with Corollary \ref{cor:P_dimension} gives the following simplification for computing exact dimensions:

\begin{corollary}\label{cor:Q0_dimension}
Suppose there are constants $c, d < 1$ such that, with $Q_0$ probability one,
\begin{align*}
\lim_{n \to \infty} \frac{\log |\varphi_n(1; \bfalpha)|^{-2}}{\log n} = c, \quad \lim_{n \to \infty} \frac{\log |\psi_n(1; \bfalpha)|^{-2}}{\log n} = d.
\end{align*}
Then $\mu_{\alpha}$ has exact dimension $2(1-c)/(2-c-d)$ for $Q$-a.e. realizations of the Verblunsky coefficients.
\end{corollary}

\begin{proof}
By Corollary \ref{cor:P_dimension} it is sufficient to show that $s_0(\mu_{\bfalpha}, \theta) = 2(1-c)/(2-c-d)$ holds $\P$ almost surely. By a rotation of the circle, it is clear that for any sequence $\bfalpha$ of Verblunsky coefficients there is the identity $s_0(\mua, \th) = s_0(\mua(e^{i \theta} \cdot), 0)$. By Lemma \ref{lemma:rotated_Verblunskies} we have that
\[ 
\mua(e^{i \theta} \cdot) = \mu_{\{ e^{i (n+1) \theta} \alpha_n \}}.
\]
Therefore by part (iv) of Proposition \ref{prop:Q_theta_properties}, one has
\[
\P(s_0(\mu_{\bfalpha}, \theta) \leq s) = \P (s_0(\mua(e^{i \theta} \cdot), 0) \leq s) = Q_0(s_0(\mu_{\bfalpha}, 0) \leq s) 
\]
for arbitrary $s$, and the result then follows from Corollary \ref{corollary:local_dim_polynomials}.
\end{proof}

\subsection{The Markov Chain Description for the Conditional Verblunskies}

In this section, we analyze the behavior of the Verblunsky coefficients under the measure $Q_0$. The description turns out to be simplest in terms of the modified Verblunsky coefficients $\gamma_n$, which maintain their independence properties but lose their rotation invariance. The radial part of each Verblunsky maintains the same law but its angular part is biased towards $z = 1$. Translating these results back into the original Verblunsky coefficients shows that the quantity $(\alpha_n, B_n(1))$ forms a Markov chain with a straightforward transition probability.

\begin{theorem}\label{theorem:markov_chain_description}
 Under $Q_0$ the modified Verblunsky coefficients $\{ \gamma_n : n \geq 0 \}$ are mutually independent with marginal laws
\begin{align}
 Q_0(\gamma_n \in \rmd z) = P_{\D}(z, 1) Q(\gamma_n \in \rmd z).
\end{align}
Consequently, under $Q_0$ the process $n \mapsto (\alpha_n, B_n(1))$ forms a Markov chain with transition kernel
\begin{align}
Q_0(\alpha_{n+1} \in \rmd z | \alpha_n, B_n(1)) = P_{\D}(\alpha_n B_n(1),1) Q(\alpha_{n+1} \in \rmd z), 
\end{align}
while $B_n(1)$ updates by formula \eqref{eqn:V_recursion}.
\end{theorem}

\begin{proof}
 This is a straightforward consequence of the independence and rotation invariance of the modified Verblunsky coefficients under $Q$ and formula \eqref{eqn:phi_psi_prod_formula} at $\theta = 0$, which shows that
\[
|\vphi_n(1)|^{-2} = \prod_{k=0}^{n-1} P_{\D}(\g_k,1), 
\]
so that the Radon-Nikodym derivative between $Q_0$ and $Q$ maintains the product structure of $Q$. The Markov chain description for the original Verblunsky coefficients follows from \eqref{eqn:phi_psi_mod_recursion} expressed in terms of $\alpha_n$ and $B_n(1)$, or alternatively by taking the result for $\gamma_n$ and applying the bijection from modified Verblunkies back to the original ones.
\end{proof}

\subsection{Coupling $Q$ and $Q_0$}

In this section we derive an algorithm that converts Verblunskys distributed according to $Q$ into Verblunskys distributed according to $Q_0$. We use it to derive several distributional properties of the Verblunskys under $Q_0$.

\begin{proposition}\label{prop:coupling_formula}
Given a sequence $\{\gamma_n \}_{n=0}^{\infty} \in \D^{\infty}$ define a new sequence $\{ \gamma_n^* \}_{n=0}^{\infty}$ by
\begin{align}\label{eqn:coupling_formula}
\gamma_n^* := \gamma_n \frac{1 + \overline{\gamma_n}}{1 + \gamma_n}.
\end{align}
If $\{ \gamma_n \}_{n=0}^{\infty}$ has law $Q$, then $\{ \gamma_n^* \}_{n=0}^{\infty}$ has law $Q_0$.  Consequently, under $Q_0$ the modified Verblunsky coefficients are invariant under conjugation, have radial laws that are the same as they are under $Q$, and have integer moments given by
\begin{align}
 \E_{Q_0}[\gamma_n^m] = \E_Q[(\gamma_n^*)^m] = \E_Q[|\gamma_n|^{2m}].
\end{align}
\end{proposition}

In the above we of course think of the $\gamma_n$ as the modified Verblunksy coefficients. It is straightforward to modify this result to obtain the analogous result for the original Verblunsky coefficients, if so desired. The proof relies on the following simple lemma.

\begin{lemma}\label{lemma:rotation_transform}
Let $Z$ be a random variable taking values in $\D$ whose law $d \Lambda(z)$ is rotationally invariant. Then, for fixed $w \in \bd$, the random variable
\begin{align*}
|Z| \phi_{-|Z|w} \left( \frac{Z}{|Z|} \right) = Z \frac{1 + \bar{Z}w}{1 + Z \bar{w}}
\end{align*}
has law $P_{\D}(z,w) \, d \Lambda(z)$ on $\D$.
\end{lemma}

\begin{proof}
Represent $Z$ by its radial and angular parts via the pair $(|Z|, Z/|Z|)$. The rotation invariance of $Z$ implies that the joint measure of this pair can be written as $\rmd \eta(r) \; \dth/2\pi$, where $\eta$ is the distribution of $|Z|$. Then recall that for $z_0 \in \D$, the conformal map $\phi_{z_0}$ sends the disk to itself and $z_0$ to $0$, and on the circle has derivative
\begin{align*}
|\phi_{z_0}'(e^{i \theta})| = P_{\D}(z_0, e^{i \theta}).
\end{align*}
Therefore if $X$ is a uniform random variable on $\bd$, then $\phi_{-z_0}(X)$ has law
\begin{align*}
|\phi_{z_0}'(e^{i \theta}) | \, \frac{d \theta}{2 \pi}.
\end{align*}
Conditioning on $|Z|$ and then using the identity
\begin{align*}
|\phi_{rw}'(e^{i \theta})| = P_{\D}(rw, e^{i \theta}) = P_{\D}(r \overline{w}, e^{-i \theta}) = P_{\D}(r e^{i \theta}, w) = |\phi_{r e^{i \theta}}'(w)|
\end{align*}
we get that $\phi_{-|Z|w}(Z/|Z|)$ has law
\begin{align*}
|\phi_{|Z| w}'(e^{i \th})| \, \frac{d \theta}{2 \pi} = |\phi_{|Z|e^{i \theta}}'(w)| \, \frac{d \theta}{2 \pi}.
\end{align*}
Multiplying by the marginal density of the $|Z|$ variable completes the proof.
\end{proof}

\begin{proof}[Proof of Proposition \ref{prop:coupling_formula}]
Lemma \ref{lemma:rotation_transform} and Theorem \ref{theorem:markov_chain_description} combine to show that the law of $\gamma_n^*$ under $Q$ is exactly $Q_0$. Conjugation invariance of $\gamma_n$ under $Q$ (which follows from rotation invariance) therefore implies conjugation invariance of $\gamma_n^*$, by \eqref{eqn:coupling_formula}. That the radial laws are the same under $Q$ and $Q_0$ follows from $|\gamma_n| = |\gamma_n^*|$. Finally, for the moment formula, use that
\[
 \E_{Q_0} \left[ \gamma_n^m \right] = \E_{Q}[(\gamma_n^*)^m] = \E_Q \left[ \left( \gamma_n \frac{1 + \overline{\gamma_n}}{1 + \gamma_n} \right)^m \right] = \E_Q \left[ |\gamma_n|^m \phi_{-|\gamma_n|} \left( \frac{\gamma_n} {|\gamma_n|} \right)^m \right].
\]
Since the function $z \mapsto \phi_{-|\gamma_n|}(z)^m$ is analytic on $\overline{\D}$ and $\gamma_n/|\gamma_n|$ is uniformly distributed on $\bd$ and independent of $|\gamma_n|$ it follows that
\begin{align*}
\E_Q \left[ \phi_{-|\gamma_n|} \left( \left. \frac{\gamma_n}{|\gamma_n|} \right)^m \right| |\gamma_n| \right] = \phi_{-|\gamma_n|}(0)^m = |\gamma_n|^m.
\end{align*}
The tower property of conditional expectation finishes the computation.
\end{proof}

%

\section{Almost Sure Local Dimensions \label{sec:local_dim}}

\subsection{Main result}

In this section we compute exact local dimensions for $\mu_{\bfalpha}$ at typical points of the measure. The techniques we use are very general and apply to \emph{any} choice of independent and rotationally invariant Verblunsky coefficients satisfying \eqref{eq:2ndmoment} and \eqref{eq:3rdmoment}, to wit
\begin{equation} \label{eq:2ndmoment2}
\E_{Q}[|\alpha_n|^2] \sim \frac{2}{\beta n}
\end{equation}
and
\begin{equation} \label{eq:3rdmoment2}
\E_{Q}[|\alpha_n|^3] = \ofrac1n.
\end{equation}
We shall prove Theorem \ref{th:exactdim}, namely that $Q$-a.s., $\mua$ has exact Hausdorff dimension $1-2/\b$. By Corollaries \ref{corollary:local_dim_polynomials} and \ref{cor:Q0_dimension} it is enough to prove the following.
 
\begin{proposition} \label{prop:Q0_circ_beta_results}
With $Q_0$-probability one
\begin{align*}
\lim_{n \to \infty} \frac{\log |\varphi_n(1)|^{-2}}{\log n} = \frac{2}{\beta}, \quad \lim_{n \to \infty} \frac{\log |\psi_n(1)|^{-2}}{\log n} = -\frac{2}{\beta}.
\end{align*}
\end{proposition}

Note that the assumption of $\beta > 2$ is only used in this section to satisfy the assumption of Corollary \ref{corollary:local_dim_polynomials} that $c < 1$. All subsequent results hold for all $\beta > 0$.

%

\subsection{Simplification to Bounded Verblunsky Coefficients \label{sec:bounded_Ver}}

In all the proofs, we will in fact assume that the $\a_n$ are uniformly bounded away from 1, i.e., for some $\dl > 0$,
\begin{equation} \label{eq:bddfrom1}
Q \left ( \forall n \geq 0 \; |\a_n| \leq 1 - \dl \right ) = 1.
\end{equation}
Here is why this is enough to obtain the results for general variables $\a_n$ satisfying \eqref{eq:2ndmoment2} and \eqref{eq:3rdmoment2}. First note that \eqref{eq:2ndmoment2} implies that the sequence $|\a_n|$ converges to zero in probability. Since it also take values in $\D$, for any $\eps > 0$ we may find $\dl > 0$ such that
\[
Q \left ( \forall \, n \geq 0 \; |\a_n| \leq 1 - \dl \right ) \geq 1 - \eps.
\]
Then, truncate the variables at radius $1-\dl$ by letting
\[
\wa_n = (|\a_n| \wedge 1 - \dl) \frac{\a_n}{|\a_n|}.
\] 
Our choice of $\dl$ means that
\[
Q \left ( \forall \, n \geq 0 \; \a_n = \wa_n \right ) \geq 1 - \eps.
\]
The variables $(\wa_n)$ are still rotationally invariant and independent and satisfy \eqref{eq:bddfrom1} and \eqref{eq:3rdmoment2}. Moreover, by H\"{o}lder's and Markov's inequalities,
\begin{align*}
\E_Q \left [ |\a_n|^2 \1{|\a_n| > 1 - \dl} \right ] & \leq \E \left [ |\a_n|^3 \right ]^{2/3} \P \left ( |\a_n| > 1 - \dl \right )^{1/3} \\
& \leq \E \left [ |\a_n|^3 \right ]^{2/3}  \E[|\a_n|^2]^{1/3} (1-\delta)^{-2/3} \\
& = o(n^{-2/3}) n^{-1/3} = o(n^{-1}),
\end{align*}
so that
\[
\E_Q \left [ |\wa_n|^2 \right ] \geq \E_Q \left [ |\a_n|^2 \one{|\a_n| \leq 1 - \dl} \right ] = \E_Q \left [ |\a_n|^2 \right ] - \E_Q \left [ |\a_n|^2 \one{|\a_n| > 1 - \dl} \right ] = \frac{2}{\b n} + o(n^{-1}).
\]
Combining this with the obvious inequality $\E_Q[|\wa_n|^2] \leq \E_Q[|\a_n|^2]$ we have
\[
\E_Q \left ( |\wa_n|^2 \right ) \sim \frac{2}{\b n}.
\]
So the second moments of the $\wa_n$ and of $\a_n$ have the same asymptotic rate of decay, with the same constant. So once we prove, for instance, that $\mu_{\widetilde{\bfa}}$ has a.s. exact Hausdorff dimension $1-2/\b$, then $\mu_{\bfa}$ has exact Hausdorff dimension $1-2/\b$ with probability at least $1-\eps$. Since this holds for all $\eps > 0$, and the result does not depend on $\eps$, then this also proves it for $\mu_{\bfa}$. 

The main purpose of assuming \eqref{eq:bddfrom1} is that it allows us to make Taylor expansions of functions of $\a_n$ that are uniform in $\bfa$. For instance, we have the Taylor series expansion
\begin{align}\label{eqn:log_Poisson_expansion}
\log P_{\D}(z,1) = 2 \Re z - 2 (\Im z)^2 + O(|z|^3),
\end{align}
when $z \to 0$. Applied to formula \eqref{eqn:phi_psi_prod_formula} for $|\vphi_n(1)|^{-2}$, this gives
\begin{equation}\label{eqn:log_circ_beta_formula}
\log |\varphi_n(1)|^{-2} = \sum_{k=0}^{n-1} \left ( 2 \Re \gamma_k - 2 (\Im \gamma_k)^2 + O(|\gamma_k|^3) \right ).
\end{equation}
The moment assumptions imply that the first two terms grow like a multiple of $\log n$ (see below), but the last one is more tricky to deal with. For instance, without knowing more, we cannot take an expectation of this formula. More precisely, we cannot put the expectation inside the $O(\cdot)$. But if the $\a_n$ do satisfy \eqref{eq:bddfrom1}, then $O(\cdot)$ simply means that there is a constant $C$, that depends only on $\dl$, such that
\[
O(|\a_k|^3) \leq C |\a_k|^3.
\]
Therefore
\[
\E \left [ \sum_{k=0}^{n-1} O(|\a_k|^3) \right ] \leq C \sum_{k=0}^{n-1} \E(|\a_k|^3) = o(\log n),
\]
by \eqref{eq:3rdmoment2}. As a conclusion, these terms will not matter.


\subsection{A lemma on martingales}

For the precise asymptotics of the first two terms in the Taylor expansion above it will be enough to analyze their means, as a result of the following result, which is taken from  \cite[Chapter 12]{Williams}.

\begin{lemma} \label{lem:martvar}
Let $M_n$ be a martingale with $M_0 = 0$ and with increments $X_n = M_n - M_{n-1}$ that satisfy $\E[X_n^2] < \infty$ for all $n$. Assume that $(b_n)$ is a sequence of positive real numbers s.t.
\[
\sum_{n=1}^{\infty} \frac{\E(X_n^2)}{b_n^2} < \infty.
\]
Then $M_n/b_n \to 0$ almost surely as $n \to \infty$.
\end{lemma}


\subsection{Asymptotics Under $Q$ \label{sec:Q_asymptotics}}

Before proving Proposition \ref{prop:Q0_circ_beta_results}, we quickly analyze the asymptotics of $|\varphi_n(1)|^{-2}$ and $|\psi_n(1)|^{-2}$ under $Q$. The results below and part (ii) of Proposition \ref{prop:spectral_lebesgue_decomp} imply that for these Verblunsky coefficients, the spectral measures $\mu_{\bfalpha}$ are almost surely singular with respect to the Lebesgue measure.

\begin{proposition}\label{prop:Q_circ_beta_results}
With $Q$ probability one
\begin{align*}
\lim_{n \to \infty} \frac{\log |\varphi_n(1)|^{-2}}{\log n} = - \frac{2}{\beta} =  \lim_{n \to \infty} \frac{\log |\psi_n(1)|^{-2}}{\log n}.
\end{align*}
\end{proposition}

\begin{proof}
We begin with the proof for $|\varphi_n(1)|^{-2}$. Using equation \eqref{eqn:log_circ_beta_formula} it is enough to determine the asymptotics of the first two summands. Recall that under $Q$ the modified Verblunskies $\gamma_k$ are independent and have the same law as the original Verblunskies $\alpha_k$. Therefore 
\begin{align*}
\sum_{k=0}^{n-1} \Re \gamma_k
\end{align*}
is a $Q$-martingale. The moment assumptions \eqref{eq:2ndmoment2} and \eqref{eq:3rdmoment2} show that it satisfies the conditions of Lemma \ref{lem:martvar} for $b_n = \log n$, and therefore is almost surely $o(\log n)$. For the second part, write
\begin{align*}
\sum_{k=0}^{n-1} (\Im \gamma_k)^2 = \sum_{k=0}^{n-1} \left [ (\Im \gamma_k)^2 - \E_{Q}[(\Im \gamma_k)^2] \right ] + \sum_{k=0}^{n-1} \E_{Q}[(\Im \gamma_k)^2].
\end{align*}
By construction the first summation on the right is a martingale, and again by the moment assumptions \eqref{eq:2ndmoment} and Lemma \ref{lem:martvar} it is $o(\log n)$ with probability one. For the second summation on the right, we have
\begin{align*}
\E_{Q}[(\Im \gamma_k)^2] = \frac{1}{2} \E_Q[|\gamma_k|^2] \sim \frac{1}{\beta k}.
\end{align*}
The first equality is by the rotation invariance of $\gamma_k$ and the second one is by \eqref{eq:2ndmoment2}. After collecting all signs and the necessary factors of two the proof for $|\varphi_n(1)|^{-2}$ is complete. The result for $|\psi_n(1)|^{-2}$ follows since $\psi_n(1; \bfalpha) = \varphi_n(1; -\bfalpha)$ and the sequence $-\bfalpha$ has the same law as $\bfalpha$ under $Q$.
\end{proof}

\subsection{Asymptotics Under $Q_0$ \label{sec:Q0_asymptotics}}

\subsubsection{First kind polynomials} \label{sec:firstkind}

Under $Q_0$ the proofs for $|\varphi_n(1)|^{-2}$ and $|\psi_n(1)|^{-2}$ are very different, in contrast with the computations under $Q$. We begin with $|\varphi_n(1)|^{-2}$, for which the first summand of \eqref{eqn:log_circ_beta_formula} is decomposed into
\begin{align*}
\sum_{k=0}^{n-1} \Re \gamma_k = \sum_{k=0}^{n-1} \left ( \Re \gamma_k - \E_{Q_0}[\Re \gamma_k] \right ) + \sum_{k=0}^{n-1} \E_{Q_0}[\Re \gamma_k].
\end{align*}
By construction, the first summation is a martingale under $Q_0$, and by \eqref{eq:2ndmoment2}, it satisfies the hypotheses of Lemma \ref{lem:martvar} for $b_n = \log n$, since  the law of $|\gamma_k|$ is the same under both $Q$ and $Q_0$ (see Proposition \ref{prop:coupling_formula}). Therefore we only need to consider the second summation, for which we have
\begin{align}\label{eqn:re_zeta_cond_expect}
\E_{Q_0}[\Re \gamma_k] = \Re \E_{Q_0}[\gamma_k] = \E_Q[|\gamma_k|^2],
\end{align}
the second equality coming from Proposition \ref{prop:coupling_formula}. Therefore by the moment assumptions \eqref{eq:2ndmoment2}
\begin{align*}
\lim_{n \to \infty} \frac{1}{\log n} \sum_{k=0}^{n-1} 2 \Re \gamma_k = \frac{4}{\beta},
\end{align*}
with $Q_0$-probability one. For the summation involving $(\Im \gamma_k)^2$ the same argument applies, one only needs to use the identity $2(\Im z)^2 = |z|^2 - \Re (z^2)$ and Proposition \ref{prop:coupling_formula} to compute that
\begin{align}\label{eqn:im_zeta_squared_cond_expect}
2\E_{Q_0}[(\Im \gamma_k)^2 ] = \E_{Q}[|\gamma_k|^2] - \E_{Q}[|\gamma_k|^4] \sim \frac{2}{\b k}.
\end{align}
Therefore
\begin{align*}
\lim_{n \to \infty} \frac{1}{\log n} \sum_{k=0}^{n-1} 2 (\Im \gamma_k)^2 &= \lim_{n \to \infty} \frac{1}{\log n} \sum_{k=0}^{n-1} E_Q[|\gamma_k|^2] - E_Q[|\gamma_k|^4] = \frac{2}{\beta},
\end{align*}
$Q_0$-almost surely.

\subsubsection{Second kind polynomials} \label{sec:secondkind}

We now study the second kind polynomials $\psi_n$, and we wish to show that
\[
\lim_{n \to \infty} \frac{\log |\psi_n(1)|^{-2}}{\log n} = -\frac{2}{\beta}.
\]
Recall that we will always use the assumption \eqref{eq:bddfrom1} that the $\a_n$, and thus the $\g_n$, are uniformly bounded away from 1.

To begin with, for each $z$ define a sequence of matrices $R_n(z)$ by
\[
R_n(z) =
\frac{1}{\sqrt{2}}
\begin{pmatrix}
\vphi_n(z) / \vphi_n(1) & - i \psi_n(z)/ \vphi_n(1) \\
\vphi_n^*(z) / \vphi_n^*(1) & i \psi_n^*(z) / \vphi_n^*(1)
\end{pmatrix}.
\]
The interest in $R_n(z)$ lies in the fact that
\begin{equation} \label{eq:Rnphin}
R_n^*(1) R_n(1) =
\begin{pmatrix}
1 & z_n \\
z_n & \left | \frac{\psi_n(1)}{\varphi_n(1)} \right |^2
\end{pmatrix}
\end{equation}
for some quantity $z_n$ that is irrelevant for our purposes. We want to understand the behavior of $|\psi_n(1)|$ under $Q_0$, and we already know from the previous section that $\log |\vphi_n(1)| \sim -1/\b \log n$. The Szeg\H{o} recursion \eqref{eqn:szego_recurrence_matrix} and some basic algebra shows that $R_n(z)$ obeys the recursion
\begin{equation} \label{eq:recRn}
R_{n+1}(z) = B_n Z R_n(z),
\end{equation}
where
\begin{equation}
B_n =
\begin{pmatrix}
(1 - \g_n)^{-1} & -\g_n (1 - \g_n)^{-1} \\
- \bar{\g_n}(1 - \bar{\g_n})^{-1} & (1 - \bar{\g_n})^{-1}
\end{pmatrix},
\quad Z = \begin{pmatrix} z & 0 \\ 0 & 1 \end{pmatrix}.
\end{equation}
By iterating \eqref{eq:recRn} we can write
\[
R_n(1) = B_{n-1} \cdots B_0 U, \quad
U =
\frac{1}{\sqrt{2}}
\begin{pmatrix}
1 & -i \\
1 & i
\end{pmatrix}.
\]
Since $U$ is a unitary matrix, we have
\begin{equation} \label{eq:RnCn}
R_n^*(1) R_n(1) = U^* (B_{n-1} \cdots B_0)^* (B_{n-1} \cdots B_0) U =  (C_{n-1} \cdots C_0)^*C_{n-1} \cdots C_0,
\end{equation}
where $C_k = U^* B_k U$. It is straightforward to compute that
\begin{equation} \label{eq:Cn}
C_k =
\begin{pmatrix}
1 & 2 \Im \frac{\g_k}{1 - \g_k} \\
0 & 1 + 2 \Re \frac{\g_k}{1 - \g_k}
\end{pmatrix}.
\end{equation}

\begin{remark}
The simplification in the shape of the matrices when transforming $B_n$ into $C_n$ can be understood as follows. First, $A_n$ is of the form
\[
\begin{pmatrix}
x & 1 - x \\
1 - \bar{x} & \bar{x}
\end{pmatrix},
\]
and hence corresponds to a M\H{o}bius transformation of the disk. Additionally, the vector $(1 \; 1)^t$ is an eigenvector with eigenvalue 1, which means that the point 1 on the disk is left invariant. Now, $U$ corresponds to the Cayley transformation that conformally maps the upper half-plane to the disk, and sends $\infty$ to 1, and conversely for $U^*$. The matrix $C_n = U^* A_n U$ therefore represents the image of that M\H{o}bius mapping in the upper half-plane, which therefore fixes $\infty$ (the image of $(1 \; 1)^t$). Hence, it is an affine map, and is therefore written as
\[
\begin{pmatrix}
1 & x \\
0 & y
\end{pmatrix},
\]
as verified by the above computation.
\end{remark}

If we write
\begin{equation} \label{eq:Cnxn}
C_{n-1} \cdots C_0 =
\begin{pmatrix}
1 & X_{n-1} \\
0 & Y_{n-1}
\end{pmatrix},
\end{equation}
then \eqref{eq:Cn} shows that $X_n$ and $Y_n$ satisfy the recurrence
\begin{equation} \label{eq:recxnyn}
X_{n} = X_{n-1} + 2 Y_{n-1} \Im \frac{\gamma_n}{1-\gamma_n}, \quad Y_{n} = Y_{n-1} \left(1 +2 \Re \frac{\g_n}{1-\g_n} \right).
\end{equation}
Note that from the expression \eqref{eqn:poisson_kernel_real_part} for the Poisson kernel we have the identity 
\begin{equation} \label{eqn:Yn_is_phin}
 Y_{n-1} = \prod_{k=0}^{n-1} P_{\D}(\g_k,1) = |\varphi_{n}(1)|^{-2}.
\end{equation}
Finally, \eqref{eq:Rnphin}, \eqref{eq:RnCn}, and \eqref{eq:Cnxn} imply that
\begin{equation} \label{eq:psiphixy}
\left | \frac{\psi_n(1)}{\varphi_n(1)} \right |^2 = X_{n-1}^2 + Y_{n-1}^2.
\end{equation} 
Therefore, it is enough to study the sequences $\{X_n\}$ and $\{Y_n\}$ as given by \eqref{eq:recxnyn}, under the law $Q_0$.

For $Y_n$ the analysis is simple enough: by \eqref{eqn:Yn_is_phin} and Proposition \ref{prop:Q0_circ_beta_results} that gives the asymptotics of $|\varphi_n(1)|^{-2}$ under $Q_0$, we have
\begin{equation} \label{eq:logyn}
\lim_{n \to \infty} \frac{\log Y_n}{\log n} = \frac{2}{\b}, \quad Q_0 \textrm{ almost surely.}
\end{equation}
In particular, for fixed $\eps, \eta > 0$, we have a constant $C > 0$ such that
\[
Y_k \leq C k^{2/\b + \eta/2}
\]
for all $k$, with $Q_0$ probability at least $1-\eps$.

Now, for the asymptotic behavior of $X_n$, by \eqref{eq:recxnyn} it is clearly enough to study
\begin{align}\label{eqn:non_truncated_martingale}
\sum_{k=1}^n Y_{k-1} \Im \frac{\g_k}{1-\g_k}
\end{align}
under $Q_0$. In fact we will study the truncated quantity 
\begin{equation}\label{eqn:truncated_martingale_defn}
 M_n = \sum_{k=1}^n (Y_{k-1} \wedge C k^{2/\beta + \eta/2}) \Im \frac{\g_k}{1-\g_k}.
\end{equation}
By the remark above we have that \eqref{eqn:non_truncated_martingale} and \eqref{eqn:truncated_martingale_defn} are equal with probability at least $1-\epsilon$. Moreover, the coupling relationship of Proposition \ref{prop:coupling_formula} gives us the identity
\begin{equation}\label{eq:vn}
 \E_{Q_0} \left[ \Im \frac{\g_n}{1-\g_n} \right] = \E_Q \left[ \frac{\Im \g_n}{1 - |\g_n|^2} \right] = 0,
\end{equation}
with the last equality following from the rotation invariance of $\g_n$ under $Q$. Hence \eqref{eqn:non_truncated_martingale} and \eqref{eqn:truncated_martingale_defn} are both mean zero martingales, by the measurability of $Y_{k-1}$ with respect to $\F_{k-1}$. The variance of each term of \eqref{eqn:truncated_martingale_defn} is bounded by
\begin{equation}
 \E_{Q_0} \!\! \left[ (Y_{k-1} \wedge C k^{2/\beta + \eta/2})^2 \left( \Im \frac{\g_k}{1-\g_k} \right)^2 \right] \leq C^2 k^{4/\beta + \eta} \E_{Q} \!\! \left[ \left( \frac{\Im \g_n}{1 - |\g_n|^2} \right)^2 \right] \leq C' k^{4/\beta + \eta - 1},
\end{equation}
where we use \eqref{eq:bddfrom1}. Now, an appeal to Lemma \ref{lem:martvar} with $b_k = k^{2/\beta + \eta}$ allows us to conclude that
\[
\frac{M_n}{n^{2/\beta + \eta}} \to 0, \quad Q_0-\textrm{a.s.}
\]
Since $M_n$ is equal to \eqref{eqn:non_truncated_martingale} with probability at least $1-\epsilon$, and since the recursion \eqref{eq:recxnyn} tells us that \eqref{eqn:non_truncated_martingale} has the same asymptotics as $X_n$, we conclude that
\[
 \lim_{n \to \infty} \frac{X_n}{n^{2/\beta + \eta}} = 0
\]
with $Q_0$ probability at least $1-\epsilon$. Since this holds for arbitrary $\epsilon$ and $\eta$, we finish by using \eqref{eq:psiphixy} and \eqref{eq:logyn} to conclude that
\[
\lim_{n \to \infty} \frac{1}{\log n} \log \left | \frac{\psi_n(1)}{\vphi_n(1)} \right |^2 = \frac{4}{\b}, \quad Q_0-\mathrm{ a.s.}
\]
Since we already proved in Section \ref{sec:firstkind} that
\[
\lim_{n \to \infty} \frac{\log | \vphi_n(1) |^{-2}}{\log n} = \frac{2}{\b}, \quad Q_0-\mathrm{ a.s.},
\]
the proof of Proposition \ref{prop:Q0_circ_beta_results} follows.

%
%
%


\section{Large Deviations for the Norm \label{sec:LDP}}

Jitomirskaya-Last dimension theory requires strong control on the norm
\[
\|\varphi_{\cdot}(1)\|_n^2 = \sum_{k=0}^n |\varphi_k(1)|^2.
\]
In the previous section we obtained this control by finding precise asymptotics for $\log |\varphi_n(1)|^2$ as $n \to \infty$, using that it is a sum of independent variables (under both $Q$ and $Q_0$). From this we derived that
\[
\lim_{n \to \infty} \frac{\log \|\varphi_{\cdot}(1)\|_n^2}{\log n} = 1 + \frac{2}{\beta}, \, Q-\mathrm{a.s.}, \quad \lim_{n \to \infty} \frac{\log \|\varphi_{\cdot}(1)\|_n^2}{\log n} = 1 - \frac{2}{\beta}, \, Q_0-\mathrm{a.s.}
\]
In this section we derive a large deviations principle (LDP) for the quantity
\[
\Upsilon_n := \frac{\log \|\varphi_{\cdot}(1)\|_n^2}{\log n},
\]
i.e. the asymptotic probability that $\Upsilon_n$ takes on an atypical value as $n \to \infty$. Such probabilities are useful for estimating the probability that the spectral measure of an interval, $\mua(\theta - \epsilon, \theta + \epsilon)$, decays atypically as $\epsilon \to 0$, for $\theta$ chosen according to $\mua$ or to the Lebesgue measure. The large deviations analysis is somewhat delicate because the scale of the terms $|\varphi_k(1)|^2$ changes as $k$ increases. Moreover we need to get control on the large deviations behavior of the entire process $k \mapsto |\varphi_k(1)|^2$ in order to get control on the behavior of the norm $\|\varphi_{\cdot}(1)\|_n^2$. To this end we analyze the large deviations of the sequence of processes
\begin{align}\label{eqn:Zn1}
Z_n(t) := \frac{\log |\varphi_{k_n(t)}(1)|^{-2}}{\log n}, \quad t \in [0,1],
\end{align}
as random elements of the Skorohod space $\D([0,1])$, for an appropriately chosen time scale $k_n$. We obtain a functional LDP for the sequence $(Z_n)$ and convert it into an LDP for $(\Upsilon_n)$, using the fact that $\|\varphi_{\cdot}(1)\|_n^2$ is a functional of $(|\varphi_k(1)|^2)$. Going from a process level LDP to one for a sequence of random numbers is an example of a \textit{contraction principle}. Process level LDPs are often easier to obtain, and that is also true in this case because $\log |\varphi_k(1)|^{-2}$ is the sum of independent random variables. The process level LDP for $(Z_n)$ is also helpful for understanding how the Verblunsky coefficients behave when they produce an atypical value of $(\Upsilon_n)$.

\subsection{Process Level LDP for Sums of Independent Variables}

The standard process level result for sums of \emph{i.i.d.} variables is Mogulskii's theorem, see \cite[Chapter 5]{DZ:LD}. Here is a quick summary: assume the iid random variables $X_1, X_2, \ldots$, taking values in $\R$, satisfy that
\[
\Lambda(\lambda) = \log \E[e^{\lambda X_i}]
\]
is finite for all $\lambda \in \R$. Define the process $S_n : [0,1] \to \R$ by
\[
S_n(t) = \frac{1}{n} \sum_{i=1}^{nt} X_i,
\] 
which we consider as a random element of the Skorohod space $\D([0,1])$. Then Mogulskii's theorem states that the sequence of probability laws induced by $S_n$ satisfies an LDP with speed $n$ and rate function
\begin{align}\label{eqn:Mogulskii_rate_fcn}
I(g) = \begin{cases}
\int_0^1 \Lambda^*(g(t)) \, dt, & \textrm{if $g$ is absolutely continuous and $g(0) = 0$,} \\
\pinf, & \textrm{otherwise,}
\end{cases}
\end{align}
where $\Lambda^*$ is the Fenchel-Legendre transform of $\Lambda$ defined by
\[
\Lambda^*(x) = \sup_{\lambda \in \R} \left \{ \lambda x - \Lambda(\lambda) \right \}.
\]
Recall that the precise statement of the LDP is that for every Borel set $A \subset \D([0,1])$, we have
\[
-\inf_{g \in A^{\circ}} I(g) \leq \liminf_{n \to \infty} \frac{1}{n} \log \P(S_n \in A) \leq \limsup_{n \to \infty} \frac{1}{n} \log \P(S_n \in A) \leq -\inf_{g \in \bar{A}} I(g)
\] 
where $A^{\circ}$ and $\bar{A}$ are the interior and closure of $A$, respectively, in the Skorohod topology. 

The process level large deviations principle for $\log |\varphi_n(1)|^{-2}$ does not fall exactly into this framework, since by equation \eqref{eqn:phi_psi_prod_formula} we have
\[
\log |\varphi_n(1)|^{-2} = \sum_{k=0}^{n-1} \log P_{\D}(\gamma_k, 1).
\]
The $\gamma_k$ are still independent (under both $Q$ and $Q_0$) but the distribution of the $\gamma_k$ is changing and concentrating around zero as $k \to \infty$. Thus we need a version of Mogulskii's theorem in which the random variables $X_i$ are independent but not necessarily identically distributed. This should be a standard result but we were unable to find it in the literature, so we spend the rest of this section deriving it. Most of the proof is a modification of the one for Mogulskii's theorem in \cite[Chapter 4]{FK:LD}, with some modifications from \cite[Chapter 5]{DZ:LD}. The statement is the following.

\begin{theorem}\label{thm:sum_LDP}
Let $(X_k)$ be a sequence of independent random variables taking values in $\R$ such that
\[
\Lambda_k(\lambda) := \log \E[ e^{\lambda X_k}]
\] 
is finite in an open interval around zero. Moreover, assume there is a sequence of positive numbers $(c_k)$ such that
\[
\Lambda(\lambda) := \lim_{k \to \infty} \frac{1}{c_k} \Lambda_k(\lambda)
\]
exists and is finite for every $\lambda \in \R$. Further assume that $(c_k)$ and $\Lambda$ satisfy
\begin{enumerate}[(a)]
\item $\log n = O(K_n)$ and $\max_{k \leq n} c_k = o(K_n)$, where $K_n := c_1 + \ldots + c_n$;
\item the function $\Lambda$ is differentiable and steep, the latter meaning that
\[
\lim_{\lambda \to \pm \infty} \Lambda'(\lambda) = \infty.
\]
\end{enumerate}
Now, for $n \geq 1$ and $1 \leq k \leq n$, define a mesh $t_{k,n}$ in $[0,1]$ by
\[
t_{k,n} = (c_1 + \ldots + c_k)/K_n,
\]
and $t_{0,n} = 0$. Then define a time scale $k_n : [0,1] \to \{1,\ldots,n \}$ by $k_n(t) = k$ for $t \in [t_{k-1,n}, t_{k,n})$, and consider the partial sum process
\begin{align}\label{eqn:Zn2}
Z_n(t) = \frac{1}{K_n} \sum_{i=1}^{k_n(t)} X_i, \quad t \in [0,1].
\end{align}
Then the processes $Z_n$ satisfy a LDP with speed $K_n$ and rate function \eqref{eqn:Mogulskii_rate_fcn}, in the topology of uniform convergence.
\end{theorem}

Note that the process $Z_n(t)$ of \eqref{eqn:Zn1} is just a special case of \eqref{eqn:Zn2} with $X_k = \log P_{\D}(\gamma_k,1)$ and $c_k$ still to be determined; we will later see that the appropriate choice is $c_k = 1/k$. For many choices of Verblunsky coefficients and $k$ small it will often be true that
\[
\E[\P_{\D}(\gamma_k, 1)^{\lambda}] = \infty,
\]
at least for $\lambda$ sufficiently large, which is why we do not assume that the $\Lambda_k$ are finite for all $\lambda$. However as long as the $\Lambda_k$ eventually become finite at each $\lambda$ and converge to some limiting value there are no difficulties. We will verify that the conditions of Theorem \ref{thm:sum_LDP} hold for a broad class of Verblunsky coefficients in the next section. The steepness condition on $\Lambda$ ensures that its Fenchel-Legendre transform $\Lambda^*$ is a good convex rate function (see \cite[Theorem 2.3.6]{DZ:LD}), which is implicitly used in the proofs to come. Also note that the topology of the theorem is uniform convergence, even though the processes $Z_n$ have jumps at the times $t_{k,n}$ and are therefore in the Skorohod space $\D([0,1])$. Thus one of the conclusions of the theorem is that the jumps can be ignored, or the process can be made continuous by linear interpolation without any modification to the LDP. We will use the latter fact later to obtain the LDP for $\Upsilon_n$.

We break the proof of Theorem \ref{thm:sum_LDP} into three lemmas. In the first, Lemma \ref{lem:exp_tightness}, we show that the processes $Z_n$ are exponentially tight in the Skorohod space $\D([0,1])$, meaning that the probability of the process not being in a compact set decays faster than any exponential in $K_n$. In the second lemma, we show that the $Z_n$ are $C$-exponentially tight, meaning that they are asymptotically continuous in a strong enough sense. This allows us to prove the LDP in the topology of uniform convergence rather than just the Skorohod topology. The final lemma proves an LDP for finite dimensional distributions of the process, and then combined with exponential tightness, this allows us to infer the process-level LDP. The latter is essentially a version of the Dawson-G\"{a}rtner theorem, although we use the approach of \cite[Chapter 4.7]{FK:LD}. In these lemmas we need the following straightforward extension of the Stolz-Ces\`{a}ro theorem, whose proof is an easy exercise.

\begin{lemma}\label{lem:cesaro}
Assume that $(c_k)$ is a sequence of positive numbers with $K_n = c_1 + \ldots + c_n \to \infty$. Assume that $(d_k)$ is another sequence of numbers with $d_n/c_n \to \ell \in \R$ as $n \to \infty$, and that there are integer sequences $(r_n^{\pm})$ such that 
\[
\lim_{n \to \infty} \frac{1}{K_n} \sum_{i = r_n^-}^{r_n^+} c_i = u.
\]
Then 
\[
\lim_{n \to \infty} \frac{1}{K_n} \sum_{i=r_n^-}^{r_n^+} d_i = u \, \ell.
\]
\end{lemma}


\begin{lemma}\label{lem:exp_tightness}
The laws of $Z_n$ are exponentially tight in $\D([0,1])$, meaning that for every $M < \infty$, there exists a compact set $E_{M} \subset \D([0,1])$ such that
\[
\limsup_{n \to \infty} \frac{1}{K_n} \log \P(Z_n \in E_M) < -M.
\]
\end{lemma}

\begin{proof}
The proof is based on \cite[Theorem 4.1]{FK:LD}. We first show that, for each fixed $t \in [0,1]$, the sequence $(Z_n(t))$ is exponentially tight in $\R$. Let $k_n = k_n(t)$, so that
\[
Z_n(t) = \frac{1}{K_n} \sum_{i=1}^{k_n} X_i.
\]
Then for $L > 0$, the exponential Chebyshev inequality implies that
\[
\frac{1}{K_n} \log \P(Z_n(t) > L) \leq -L + \frac{1}{K_n} \sum_{i=1}^{k_n} \Lambda_i(1).
\]
By definition of $K_n$ and the assumption that $\Lambda_i(1)/c_i \to \Lambda(1)$ as $i \to \infty$, it follows from the Stolz-C\`{e}saro type theorem in Lemma \ref{lem:cesaro} that the right hand side tends to $t \Lambda(1)$ as $n \to \infty$. The same argument produces a similar $L$ dependent bound for $\P(Z_n(t) < -L)$, and then taking $L \to \infty$ shows the exponential tightness.

Now we prove that the $Z_n$ are tight in the Skorohod space $\D([0,1])$. As in \cite[Theorem 4.1]{FK:LD} this is done by using the pointwise exponential tightness above and proving that
\[
\lim_{\delta \to 0} \lim_{n \to \infty} \frac{1}{K_n} \log \P(\omega'(Z_n, \delta) > \epsilon) = -\infty,
\] 
where $\omega' : \D([0,1]) \times \R \to [0, \infty]$ is the modulus-of-continuity type object defined by
\[
\omega'(f,\delta) = \inf_{(s_j)} \max_j \sup_{s_{j-1} \leq s,t \leq s_j} |f(s) - f(t)|, 
\]
and the infimum is over all partitions of $[0,1]$ with mesh size (minimum length of an interval) greater than $\delta$. We will only need to consider the partition with all intervals the same size and of length $\delta$, except for potentially the last one having longer length in $[\dl, 2 \dl)$. For $s < t$ we have by definition of $Z_n$ that
\[
Z_n(t) - Z_n(s) = \frac{1}{K_n} \sum_{i = k_n(s) + 1}^{k_n(t)} X_i.
\]
Therefore, by using the equisized partition mentioned above we have
\[
\omega'(Z_n, \delta) \leq \max_{0 \leq j < 1/\delta} \frac{1}{K_n} \left| \sum_{i=k_n(j \delta) + 1}^{k_n((j+1)\delta)} X_i \right|.
\]
We then apply the standard union bound and replace the summation with the maximum of the summands to obtain
\[
\P(\omega'(Z_n, \delta) > \epsilon) \leq \delta^{-1} \max_{0 \leq j < 1/\delta} \P \left( \frac{1}{K_n} \left| \sum_{i=k_n(j \delta) + 1}^{k_n((j+1)\delta)} X_i \right| > \epsilon \right).
\]
Each of the probabilities above can be bounded by the exponential Chebyshev inequality to obtain
\[
\P(\omega'(Z_n, \delta) > \epsilon) \leq \delta^{-1} e^{-\lambda K_n \epsilon} \max_{0 \leq j < 1/\delta} \max_{\pm \lambda} \prod_{i=k_n(j \delta) + 1}^{k_n((j+1)\delta)} \E e^{\pm \lambda X_i},
\]
where the inner maximum means the larger over the indicated terms with $\lambda$ or $-\lambda$. Therefore by taking logarithms we have
\begin{align}\label{eqn:bound1}
\frac{1}{K_n} \log \P(\omega'(Z_n, \delta) > \epsilon) \leq -\frac{\log \delta}{K_n} - \lambda \epsilon + \max_{0 \leq j < 1/\delta} \max_{\pm \lambda} \frac{1}{K_n} \sum_{i=k_n(j \delta) + 1}^{k_n((j+1)\delta)} \Lambda_i(\pm \lambda).
\end{align}
However, by definition of $K_n$ and the functions $k_n$, we have
\[
\lim_{n \to \infty} \frac{1}{K_n} \sum_{i=k_n(j \delta) + 1}^{k_n((j+1)\delta)} \!\!\!\! c_i = (j+1) \delta - j \delta = \delta,
\]
and since $\Lambda_i(\pm \lambda)/c_i \to \Lambda(\pm \lambda)$ as $i \to \infty$, the Stolz-C\`{e}saro theorem of Lemma \ref{lem:cesaro} implies that
\[
\lim_{n \to \infty} \frac{1}{K_n} \sum_{i=k_n(j \delta) + 1}^{k_n((j+1)\delta)} \Lambda_i(\pm \lambda) = \delta \, \Lambda(\pm \lambda).
\]
Therefore by taking the limsup of \eqref{eqn:bound1} we obtain
\[
\limsup_{n \to \infty} \frac{1}{K_n} \sum_{i=k_n(j \delta) + 1}^{k_n((j+1)\delta)} \Lambda_i(\pm \lambda) \leq - \lambda \epsilon + \delta \max_{\pm \lambda} \Lambda(\pm \lambda) = -\delta \min_{\pm \lambda} \{ \lambda \epsilon/\delta - \Lambda(\pm \lambda) \}.
\]
This is true for all $\lambda > 0$, therefore we can optimize the right hand side over $\lambda$ to obtain
\begin{align}\label{eqn:mod_cont_bound}
\limsup_{n \to \infty} \frac{1}{K_n} \log \P(\omega'(Z_n, \delta) > \epsilon) \leq -\delta \min_{\pm \lambda} \Lambda^*(\pm \epsilon/\delta).
\end{align}
Then the steepness condition $\Lambda'(\lambda) \to \infty$ as $\lambda \to \pm \infty$ implies that $\delta \Lambda^*(\pm \epsilon/\delta) \to \infty$ as $\delta \to 0$, so the above exactly shows that the $Z_n$ are exponentially tight in the Skorohod space $\D([0,1])$.

\end{proof}

\begin{lemma}\label{lem:exp_equivalence}
The processes $Z_n$ are $C$-exponentially tight in $D([0,1])$, meaning that for every $\delta > 0$
\[
\limsup_{n \to \infty} \frac{1}{K_n} \log \P \left( \sup_{t \in[0,1]} |Z_n(t) - Z_n(t-)| > \delta \right) = -\infty.
\]
\end{lemma}


\begin{proof}
By definition of $Z_n$ the only jumps are at the times $t_{k,n}$, therefore
\[
\sup_{t \in [0,1]} |Z_n(t) - Z_n(t-)| = \frac{1}{K_n} \max_{1 \leq k \leq n} |X_k|.
\]
Then by the standard union bound 
\[
\P \left( \sup_{t \in [0,1]} |Z_n(t) - Z_n(t-)| > \delta \right) \leq \sum_{k=1}^n \P(|X_k| > \delta K_n). 
\]
Now, for each fixed $\lambda > 0$, there exists an integer $M = M(\lambda)$ such that $\E[e^{\lambda |X_k|}] < \infty$ for all $k \geq M(\lambda)$, by the assumption that $\Lambda_k(\pm \lambda)/c_k$ converges to a finite quantity for each $\lambda$. Therefore on all terms with $k \geq M$ in the last summation we can apply the exponential Chebyshev inequality to obtain
\[
\sum_{k=1}^n \P(|X_k| > \delta K_n) \leq M + \sum_{k=M}^n e^{-\lambda \delta K_n} \E[ e^{\lambda |X_k|}].
\]
Taking logarithms of both sides yields (also using subadditivity of the logarithm)
\[
\frac{1}{K_n} \log \P \left( \sup_{t \in [0,1]} |Z_n(t) - Z_n(t-)| > \delta \right) \leq \frac{1}{K_n} \log M -\delta \lambda + \frac{1}{K_n} \log \sum_{k=M}^n \E[ e^{\lambda |X_k|}].
\]
By the assumption that $c_k^{-1} \Lambda_k(\pm \lambda)$ converges as $k \to \infty$ we have that there exists a constant $A = A(\lambda) > 0$ such that $\log \E[e^{\pm \lambda X_k}] \leq A c_k$ for all $k \geq M$, and therefore
\begin{align*}
\frac{1}{K_n} \log \P \left( \sup_{t \in [0,1]} |Z_n(t) - Z_n(t-)| > \delta \right)
&\leq \frac{1}{K_n} \log M -\delta \lambda + \frac{2}{K_n} \log \sum_{k=M}^n e^{A c_k} \\
&\leq \frac{1}{K_n} \log M -\delta \lambda + \frac{2}{K_n} (A \max_{1 \leq k \leq n} c_k + \log n).
\end{align*}
Now, recalling that $K_n = c_1 + \ldots + c_n$ and that we assumed that $\log n = O(K_n)$ and $\max_{k \leq n} c_k = o(K_n)$, for the last term in the expression above there is a universal constant $C > 0$ such that
\[
\limsup_{n \to \infty} \frac{1}{K_n} \log \P \left( \sup_{t \in [0,1]} |Z_n(t) - Z_n(t-)| > \delta \right) \leq -\delta \lambda + C.
\]
The latter holds for all $\lambda > 0$, so taking $\lambda \to \infty$ completes the proof.
\end{proof}

Note that to handle the case that $\Lambda_k(\lambda) = \infty$ for $k$ small required only mild modifications. In the remaining two lemmas no extra modifications beyond those used above are necessary to handle this complication, so we provide the proofs under the assumption $\Lambda_k(\lambda) < \infty$ for all $k$ and $\lambda$ and let the reader fill in the details. For the full details of how to use Lemma \ref{lem:exp_equivalence} to strengthen the topology of the LDP for $Z_n$ to that of uniform convergence, see \cite[Chapter 4.4]{FK:LD}.


\begin{lemma}\label{lem:sum_LDP_ptwise}
Let $s_1 < s_2 < \ldots s_m$ be a partition in $[0,1]$. Then the process $(Z_n(s_1), \ldots, Z_n(s_m))$ satisfies a large deviations principle in $\R^m$ with speed $K_n$ and rate function
\[
I_{s_1, \ldots, s_m}^Z(\mathbf{x}) = \sum_{j=1}^m (s_j - s_{j-1}) \Lambda^* \left( \frac{\mathbf{x}_j - \mathbf{x}_{j-1}}{s_j - s_{j-1}} \right).
\]
Consequently, $Z_n$ satisfies a LDP with speed $K_n$ and rate function \eqref{eqn:Mogulskii_rate_fcn}, in the topology of uniform convergence.
\end{lemma}

\begin{proof}[Proof of Lemma \ref{lem:sum_LDP_ptwise}]
The proof is similar to \cite[Lemma 5.1.6]{DZ:LD}, which proves the same statement for the standard situation of iid random variables. In our situation all that needs to be modified is the proof of the finite-dimensional LDP which we now give. Let $Y_n$ be the random element of $\R^m$ given by
\[
Y_n = (Z_n(s_1), Z_n(s_2) - Z_n(s_1), \ldots, Z_n(s_m) - Z_n(s_{m-1})).
\]
Further, for $n$ large let $0 = t_{k_0,n} < t_{k_1, n} < \ldots < t_{k_m,n} \leq 1$ be the ordered sequence of times that are closest to the $s_j$ in the mesh, in the sense that $s_j \in [t_{k_j,n}, t_{k_{j}+1,n})$ for all $1 \leq j \leq m$. Since the mesh size goes to zero as $n \to \infty$ (by $\max_{k \leq n} c_k = o(K_n)$), this is always possible (with the $t_{k_i,n}$ distinct) for $n$ large enough. Also note that the terms $k_j$ depend on $n$ but we suppress the dependence in the notation. Now writing $Y_n^j = Z_n(s_j) - Z_n(s_{j-1})$ for the coordinates of $Y_n$ we have by definition of $Z_n$ that
\[
Y_n^j = Z_n(s_j) - Z_n(s_{j-1}) = Z_n(t_{k_j,n}) - Z_n(t_{k_{j-1},n+1}) = \frac{1}{K_n} \sum_{i=k_{j-1}+1}^{k_j} X_i.
\]
Therefore, for $(\lambda_1, \ldots, \lambda_m) \in \R^m$, the exponential moment-generating function of the variable $Y_n = (Y_n^1, \ldots, Y_n^m)$ satisfies
\begin{align}\label{eqn:log_gf}
\log \E \exp \left( K_n \sum_{j=1}^m \lambda_j Y_n^j \right) = \sum_{j=1}^m \sum_{i = k_{j-1} + 1}^{k_j} \Lambda_i(\lambda_j).
\end{align}
But by definition of the series $t_{k,n}$ we have that
\[
t_{k_j,n} - t_{k_{j-1},n} = \frac{1}{K_n} \sum_{i=k_{j-1}+1}^{k_j} c_i,
\]
and since the $t_{k_j,n}$ are the closest points to $s_j$ in the mesh, we have that the above converges to $s_j - s_{j-1}$ as $n \to \infty$. Furthermore, by the definition of $\Lambda$ as a limit, we have that $\Lambda_i(\lambda_j)/c_i \to \Lambda(\lambda_j)$ as $i \to \infty$, so by combining these last two facts and Lemma \ref{lem:cesaro}, we have
\[
\lim_{n \to \infty} \frac{1}{K_n} \sum_{i = k_{j-1} + 1}^{k_j} \Lambda_i(\lambda_j) = (s_j - s_{j-1}) \Lambda(\lambda_j).
\]
Combining this with \eqref{eqn:log_gf} we obtain that
\[
\lim_{n \to \infty} \frac{1}{K_n} \log \E \exp \left( K_n \sum_{j=1}^m \lambda_j Y_n^j \right) = \sum_{j=1}^{m} (s_j - s_{j-1}) \Lambda(\lambda_j).
\]
By the assumptions on $\Lambda$ the right hand side is, as a function of $(\lambda_1, \ldots, \lambda_m) \in \R^m$, differentiable, lower semi-continuous, and steep (in any direction going to infinity in $\R^m$). Therefore, by the G\"{a}rtner-Ellis theorem \cite[Theorem 2.3.6]{DZ:LD}, the processes $Y_n$ satisfy an LDP as random elements of $\R^m$, with speed $K_n$ and rate function
\[
I_{s_1, \ldots, s_m}^Y(\mathbf{x}) = \sum_{j=1}^m (s_j - s_{j-1}) \Lambda^* \left( \frac{\mathbf{x}_j}{s_j - s_{j-1}} \right)
\]
for $\mathbf{x} \in \R^m$. Finally, since the mapping $Y_n = (Z_n(s_1), Z_n(s_2) - Z_n(s_1), \ldots, Z_n(s_m) - Z_n(s_{m-1})) \mapsto (Z_n(s_1), \ldots, Z_n(s_m))$ is clearly continuous and bijective, the LDP for $Y_n$ translates into an LDP for $Z_n$ with speed $K_n$ and rate function
\[
I_{s_1, \ldots, s_m}^Z(\mathbf{x}) = \sum_{j=1}^m (s_j - s_{j-1}) \Lambda^* \left( \frac{\mathbf{x}_j - \mathbf{x}_{j-1}}{s_j - s_{j-1}} \right),
\]
by means of the contraction principle \cite[Theorem 4.2.1]{DZ:LD}. This completes the LDP for the processes $(Z_n(s_1), \ldots, Z_n(s_m))$. To extend to the process-level LDP, use the $C$-exponential tightness of Lemma \ref{lem:exp_tightness} and \cite[Theorem 4.30]{FK:LD} to conclude that $Z_n$ satisfies the LDP with good rate function
\[
I(g) = \sup_{\{ s_i \}} I_{s_1, \ldots, s_m}(g(s_1), \ldots, g(s_m)),
\]
where the supremum is over all partitions in $[0,1]$ (for any value of $m$). To conclude that this supremum is equal to the rate function $I(g)$ of \eqref{eqn:Mogulskii_rate_fcn} is a straightforward calculus exercise, see \cite[Corollary 5.1.10]{DZ:LD} for full details.  
\end{proof}

\subsection{LDP for $Z_n$}

Now we apply the results of the last section to the orthogonal polynomials. The next lemma verifies the assumptions of Theorem \ref{thm:sum_LDP} for the specific partial sum process of \eqref{eqn:Zn1}, and identifies the corresponding rate functions.

\begin{lemma}\label{lem:Verblunkies_good}
Assume that under the measure $Q$, the Verblunsky coefficients $\alpha_k$ are independent and rotation invariant with the usual second moment assumption \eqref{eq:2ndmoment}. Further assume that for some $\epsilon > 0$ and for all $\kappa > 0$, the radial parts satisfy
\[
\E[|\alpha_k|^3] = O(k^{-1-\epsilon}), \quad \limsup_{k \to \infty} \E[(1 - |\alpha_k|)^{-\kappa}] < \infty.
\]
Then $X_k = \log P_{\D}(\gamma_k, 1)$ satisfies the hypotheses of Theorem \ref{thm:sum_LDP} with $c_k = 1/k$ and
\[
\Lambda^*(x) = \frac{\beta}{8} \left( x + \frac{2}{\beta} \right)^2.
\]
Furthermore, under $Q_0$, the same $X_k$ satisfy the hypotheses also with $c_k = 1/k$ and rate function 
\[
\Lambda^*(x) - x = \frac{\beta}{8} \left( x- \frac{2}{\beta} \right)^2.
\]
\end{lemma}

\begin{proof}
First note that rotation invariance of the $\alpha_k$ implies that the modified Verblunskies $\gamma_k$ are independent (Lemma \ref{lemma:KS_modified_Verblunsky_law}), but that is all we use the rotation invariance for. After that we only need to check that for each $\lambda \in \R$ the limit
\[
\lim_{k \to \infty} \frac{1}{c_k} \log \E[\exp \left \{ \lambda \log P_{\D}(\gamma_k, 1) \right \}] = \lim_{k \to \infty} k \log \E[P_{\D}(\gamma_k, 1)^{\lambda}] =: \Lambda(\lambda)
\]
exists and is finite, and that the resulting $\Lambda$ is differentiable and steep. First work on the event that $|\alpha_k| \leq 1/2$, where the truncation allows us to use the expansion \eqref{eqn:log_Poisson_expansion} to write
\[
P_{\D}(\gamma_k, 1)^{\lambda} = 1 + 2 \lambda \Re \gamma_k - 2 \lambda (\Im \gamma_k)^2 + 2 \lambda^2 (\Re \gamma_k)^2 + O(|\gamma_k|^3),
\]
and the constants in the $O$ term are uniform in $\alpha_k$. We already saw in Section \ref{sec:bounded_Ver} that
\[ 
\E[|\gamma_k|^2 \1{|\gamma_k| \leq 1/2}] = \frac{2}{\beta k} + o(k^{-1}), 
\]
so that we may take an expectation of the above asymptotic expression to obtain
\begin{align}\label{eqn:poisson_small}
\E[P_{\D}(\gamma_k, 1)^{\lambda} \1{|\gamma_k| \leq 1/2}] = 1 + \lambda (\lambda - 1) \frac{2}{\beta k} + o(k^{-1-\epsilon}).
\end{align}
Now for the case $|\alpha_k| \geq 1/2$ begin with the estimate
\[
P_{\D}(z,1) = 2 \Re \frac{1+z}{1-z} \leq \frac{1 + |z|}{1 - |z|} \leq \frac{2}{1-|z|}.
\]
Let $q \in (1,1+\epsilon)$ and $p$ satisfy $1/p + 1/q = 1$. Then from the bound above and a combination of H\"{o}lder's and Markov's inequality,
\begin{align*}
\E[P_{\D}(\gamma_k,1)^{\lambda} \1{|\gamma_k| \geq 1/2}] &\leq C \E[(1 - |\gamma_k|)^{-\lambda} \1{|\gamma_k| \geq 1/2}] \\
& \leq C \E[(1- |\gamma_k|)^{-\lambda p}]^{1/p} \P(|\gamma_k| > 1/2)^{1/q} \\
&\leq C \E[(1 - |\gamma_k|)^{-\lambda p}]^{1/p} \E[|\gamma_k|^3]^{1/q} \\
&\leq C \E[(1 - |\gamma_k|)^{-\lambda p}]^{1/p} O(k^{-(1+\epsilon)/q}),
\end{align*}
the last inequality following by the assumption on $\E[|\gamma_k|^3]$, where $C$ is a constant independent of $k$. Therefore by the negative moment assumption on $(1-|\gamma_k|)$ and the fact that $(1+\epsilon)/q > 1$ we have
\begin{align}\label{eqn:poisson_big}
\limsup_{k \to \infty} k \E[P_{\D}(\gamma_k,1)^{\lambda} \1{|\gamma_k| \geq 1/2}] = 0.
\end{align}
Combining \eqref{eqn:poisson_big} with \eqref{eqn:poisson_small} we therefore obtain
\[
\Lambda(\lambda) = \lim_{k \to \infty} k \log \E[P_{\D}(\gamma_k, 1)^{\lambda}] = \frac{2}{\beta} \lambda(\lambda-1).
\]
This $\Lambda$ is clearly differentiable and steep, and it is straightforward to verify that its Legendre-Fenchel transform is $\Lambda^*$ as given. Finally, to prove the statement under $Q_0$, recall that Theorem \ref{theorem:markov_chain_description} implies that the $\gamma_k$ are also independent under $Q_0$, and the fact that the Radon-Nikodym derivative between $Q$ and $Q_0$ (for $\gamma_k$) is $P_{\D}(\gamma_k, 1)$ gives that
\[
\lim_{k \to \infty} k \E_{Q_0} [P_{\D}(\gamma_k, 1)^{\lambda}] = \lim_{k \to \infty} k \E_Q [P_{\D}(\gamma_k, 1)^{\lambda + 1}] = \Lambda(\lambda + 1)
\]
Therefore Theorem \ref{thm:sum_LDP} also applies under $Q_0$ with rate function as given.
\end{proof}

\subsection{LDP for $\Upsilon_n$}

To obtain the LDP for $\Upsilon_n$ we write it as a function of the process $Z_n$ and then apply the contraction principle. Recall that $Z_n(t)$ is defined by \eqref{eqn:Zn1}, which can be rewritten as
\[
|\varphi_{k_n(t)}|^{-2} = n^{Z_n(t)}, \quad t \in [0,1].
\]
Now recall the definition of the norm $\|\varphi_{\cdot}\|_n^2$ as
\[
\|\varphi_{\cdot}(1)\|_n^2 = \sum_{k=0}^n |\varphi_k(1)|^2.
\]
In the remainder of this section it will be convenient to start the sum at $k=1$ instead. Clearly, this will have no effect on the large deviations principle. Now using that $k_n(t) = k$ on each interval $t \in [t_{k-1,n}, t_{k,n})$ and $t_{k,n} - t_{k-1,n} = c_k = 1/k$ in our case, we can rewrite the norm as 
\begin{align*}
\|\varphi_{\cdot}(1)\|_n^2 = \sum_{k=1}^n \frac{1}{c_k} \int_{t_{k-1,n}}^{t_{k,n}} n^{-Z_n(s)} \ds 
= K_n \sum_{k=1}^n k \int_{t_{k-1,n}}^{t_{k,n}} n^{-Z_n(s)} \ds 
&= K_n \int_0^1 k_n(s) n^{-Z_n(s)} \ds.
\end{align*}
Now by definition of $k_n(s)$ we have
\[
k_n(s) = \min \left \{ k \geq 1 : \frac{1}{K_n} \sum_{i=1}^k \frac1i \geq s \right \},
\]
from which it follows by the asymptotics of the harmonic series that $k_n(s) = C(s) n^{s + O(1/n)}$, where $C(s)$ is a non-random, positive, continuous function on $[0,1]$ that is strictly bounded away from zero. Therefore we have
\[
\|\varphi_{\cdot}(1)\|_n^2 \sim K_n \int_0^1 n^{s - Z_n(s)} \, ds, \quad n \to \infty,
\]
so it is enough to prove the LDP for
\[
\frac{1}{\log n}\log \int_0^1 n^{s - Z_n(s)} \, ds,
\]
the latter using that $K_n \sim \log n$. By Laplace's principle we expect that
\[
\frac{1}{\log n}\log \int_0^1 n^{s - Z_n(s)} \, ds \sim \max_{s \in [0,1]} \left \{ s-Z_n(s) \right \},
\]
as $n \to \infty$. In Lemma \ref{lem:max_int_exp_equivalence} we will show that these two sequences are close enough to be exponentially equivalent, meaning that if one satisfies an LDP then the other satisfies the same LDP. Therefore to complete the proof of Theorem \ref{thm:norm_LD} it will be enough to prove an LDP for the sequence
\[
\max_{s \in [0,1]} \left \{ s - Z_n(s) \right \},
\]
which we do via Theorem \ref{thm:sum_LDP} and the contraction principle.

\begin{lemma}\label{lem:max_LDP}
The sequence $\max_{0 \leq s \leq 1}\{s - Z_n(s)\}$ satisfy an LDP with speed $K_n$ and rate function $\Lambda^*(1-x)$.
\end{lemma}

\begin{proof}
Since the mapping $g \to \max_{0 \leq s \leq 1} \{s - g(s)\}$ is continuous from $C_0([0,1])$ to $\R$, it follows by Theorem \ref{thm:sum_LDP} and the contraction principle that $\max_{0 \leq s \leq 1} \{s - Z_n(s)\}$ satisfies the LDP with speed $K_n$ and rate function
\[
J(x) = \inf \left \{ \int_0^1 \Lambda^*(\phi'(t)) \dt : \max_{0 \leq s \leq 1} \{s - \phi(s)\} = x \right \}.
\]
As $\phi(0) = 0$, there is no solution for $x < 0$, and thus $J(x) = \infty$. We therefore now assume that $x \geq 0$. By taking $\phi(s) = s(1-x)$ we clearly get that $J(x) \leq \Lambda^*(1-x)$. However, this $\phi$ is also the minimizer. To see this, let us assume that $s - \phi(s)$ has its maximum at $s_0 \in (0,1)$ (with trivial modifications if $s_0 \in \{0,1\}$), so in particular $s_0 - \phi(s_0) = x$. The convexity of $\Lambda^*$ and Jensen's inequality then give
\begin{align*}
\int_0^1 \Lambda^*(\phi'(s)) \ds & = \int_0^{s_0} \Lambda^*(\phi'(s)) \ds + \int_{s_0}^1 \Lambda^*(\phi'(s)) \ds \\
& \geq s_0 \: \Lambda^* \left ( \frac{1}{s_0} \int_0^{s_0} \phi'(s) \ds \right ) + (1 - s_0) \: \Lambda^* \left ( \frac{1}{1 - s_0} \int_{s_0}^1 \phi'(s) \ds \right ) \\
& = s_0 \: \Lambda^* \left ( 1 - \frac{x}{s_0} \right ) + (1 - s_0) \: \Lambda^* \left ( \frac{1}{1 - s_0} (\phi(1) - \phi(s_0)) \right ).
\end{align*}
Now note that
\[
\frac{1}{1 - s_0} (\phi(1) - \phi(s_0)) \geq \frac{1}{1 - s_0} ( 1 - x - (s_0 - x)) = 1,
\]
and that $\Lambda^*$ (as as well as $\Lambda^*(x) - x$, for $\b > 2$, when we work under $Q_0$) is increasing on $[1,\infty)$. From this fact, the previous computations, and the convexity of $\Lambda^*$, we therefore deduce that
\[
\int_0^1 \Lambda^*(\phi'(s)) \ds \geq s_0 \: \Lambda^* \left ( 1 - \frac{x}{s_0} \right ) + (1 - s_0) \: \Lambda^*(1) \geq \Lambda^* \left ( s_0 \left ( 1 - \frac{x}{s_0} \right ) + 1 - s_0 \right ) =  \Lambda^* (1-x),
\]
which shows that $\phi(s) = s(1-x)$ is indeed the minimizer of the functional, and thus
\[
J(x) = \int_0^1 \Lambda^*(1-x),
\]
as we wanted to show.
\end{proof}

\begin{lemma}\label{lem:max_int_exp_equivalence}
The sequences $\log \int_0^1 n^{s - Z_n(s)} \ds/\log n$ and $\max_{s \in [0,1]} \{s - Z_n(s)\}$ are exponentially equivalent, meaning that for every $\delta > 0$,
\[
\limsup_{n \to \infty} \frac{1}{K_n} \log \P \left( \left| \frac{1}{\log n} \log \int_0^1 n^{s - Z_n(s)} \, ds - \max_{s \in [0,1]} \{ s - Z_n(s) \} \right| > \delta \right) = -\infty.
\]
\end{lemma}

\begin{proof}
Throughout we will write $Y_n(s) = s - Z_n(s)$. We will use the continuity properties of $Y_n$ to prove the equivalence, along the following easily derived bounds for Laplace's principle: if $f$ is a bounded, measurable function on $[0,1]$ and if
\[
\sup_{|s-t| < \delta} |f(s) - f(t)| < \epsilon
\]
for some particular $\epsilon > 0$ and $\delta \in (0,1/2)$, then
\begin{align}\label{eqn:bound4}
0 \leq \underset{s \in [0,1]}{\mathrm{ess~sup}} f(s) - \frac{1}{\log n} \log \int_0^1 n^{f(s)} \ds \leq \epsilon - \frac{\log \delta}{\log n}.
\end{align}
Of course the lower bound above is obvious. Now recall the modulus of continuity $\omega'$ from Lemma \ref{lem:exp_tightness}, and note that by its definition and that of $Y_n$ we have $\omega'(Y_n, \delta) \leq \omega'(Z_n) + \delta$. Then for fixed $\epsilon > 0$ and $\delta \in (0, \epsilon/2)$ we have, by \eqref{eqn:bound4}, that
\begin{align*}
\P \left( \underset{s \in [0,1]}{\mathrm{ess~sup}}\, Y_n(s) - \frac{1}{\log n} \log \int_0^1 n^{Y_n(s)} \ds > \epsilon - \frac{\log \delta}{\log n} \right) & \leq \P \left( \sup_{|s-t| < \delta} |Y_n(s) - Y_n(t)| > \epsilon \right) \\
& \leq \P \left( \omega'(Y_n, \delta) > \epsilon \right) \\
& \leq \P \left( \omega'(Z_n, \delta) + \delta > \epsilon \right) \\
& \leq \P \left( \omega'(Z_n, \delta) > \epsilon/2 \right).
\end{align*}
The second inequality follows by definition of $\omega'$. Therefore, for all $n$ sufficiently large such that $-\log \delta/\log n < \epsilon$,
\[
\frac{1}{K_n} \log \P \left( \underset{s \in [0,1]}{\mathrm{ess~sup}}\, Y_n(s) - \frac{1}{\log n} \log \int_0^1 n^{Y_n(s)} \ds > 2\epsilon \right)  \leq \frac{1}{K_n} \log \P(\omega'(Z_n, \delta) > \epsilon/2).
\]
Then take limsup of both sides as $n \to \infty$ and use that
\[
\lim_{\delta \to 0} \limsup_{n \to \infty} \frac{1}{K_n} \log \P(\omega'(Z_n, \delta) > \epsilon/2) = 0,
\]
as proved in \eqref{eqn:mod_cont_bound} of Lemma \ref{lem:exp_tightness} to conclude. 
\end{proof}


\end{document}